\documentclass[preprint,11pt]{amsart}
\usepackage{lipsum}
\usepackage{graphicx}
\usepackage{amsmath,leftidx,amsthm}
\usepackage[all]{xy}
\usepackage{xspace}
\usepackage{amssymb}
\usepackage[latin1]{inputenc}
\usepackage{graphicx,color}
\usepackage{hyperref,fancyhdr}
\usepackage{fullpage}
\usepackage{versions}
\usepackage{textcomp,listings}
\usepackage{enumitem}
\usepackage{tikz}
\setenumerate[1]{label=(\arabic*)}

\newtheorem{theorem}{Theorem}[section]

\newtheorem{lemma}[theorem]{Lemma}
\newtheorem{cor}[theorem]{Corollary}

\theoremstyle{definition}
\newtheorem{definition}[theorem]{Definition}

\newtheorem{proposition}[theorem]{Proposition}

\theoremstyle{remark}
\newtheorem*{remark}{Remark}

\numberwithin{equation}{section}

  \renewcommand{\tilde}{\widetilde}

    \newcommand{\Ran}{\mbox{\rm Range}}
    
    \renewcommand{\phi}{\varphi}
    \newcommand{\N}{\mathcal{N}}

   \newcommand{\D}{\mathbb{D}}
      
   \newcommand{\R}{\mathbb{R}}

    \newcommand{\BS}{\mathcal{B}}

   \newcommand{\CP}{\mathbb{C}}

   \newcommand{\Z}{\mathbb{Z}}
   \renewcommand{\l}{{\rule{0in}{0.1in}l}}
 	\newfont{\caps}{cmcsc9}  
 	\newfont{\jour}{cmti9}  
\theoremstyle{remark}

\definecolor{green}{rgb}{0,0.5,0}
\definecolor{dkgreen}{rgb}{0,0.6,0}
\definecolor{gray}{rgb}{0.5,0.5,0.5}
\definecolor{mauve}{rgb}{0.58,0,0.82}
\definecolor{purple}{rgb}{0.58,0,0.62}

\renewcommand{\phi}{\varphi}

\newcommand\myov[2]{\genfrac{}{}{0pt}{}{#1}{#2}}

\newcommand{\B}{\mathbb{B}}

\def\and{{\quad\text{and}\quad}}

\begin{document}

\title
[{Fredholm Multiplication Operators}]
{Fredholm Multiplication Operators on Banach Spaces of Analytic Functions on the Open Unit Disk}
\author{Paul Bourdon}
\address{Department of Mathematics, University of Virginia,
         Charlottesville, USA}
\email{psb7p@virginia.edu}
\author{Mahsa Fatehi}
\address{Department of Mathematics, Shi.C., Islamic Azad University, Shiraz, Iran.}
\email{fatehimahsa@iau.ac.ir;  fatehimahsa@yahoo.com}

\begin{abstract}  We identify properties of a Banach space $\BS$ of  analytic functions on the open unit disk $\D$ in the complex plane ensuring that a multiplication operator $M_\psi: \BS \to\BS$ is Fredholm if and only if its symbol $\psi$ is bounded away from $0$ near $\partial \D$. The properties we identify are shared by a wide variety of much-studied spaces, including the Hardy spaces $H^p(\D)$, weighted Bergman spaces $A^p_\omega(\D)$, Hardy-Sobolev spaces $H^2_\beta(\D)$, the spaces $S_j^p(\D)$ of functions having $j$-th derivative in $H^p(\D)$, and the disk algebra $A$.  Thus, as a corollary, our work characterizes Fredholm multiplication operators on these spaces.   In addition,  we describe the closed, finite-codimensional subspaces of $\BS$ that are invariant under  $M_z: \BS \to \BS$ in terms of the zeros that the functions in such subspaces have in common.  We discuss connections between these subspaces and the problem of characterizing the Fredholm multiplication operators on $\BS$, and we prove that $M_z$ restricted to such a subspace is always cyclic, with a polynomial cyclic vector having degree equal to the codimension of the subspace.

\end{abstract}

\maketitle

\section{Introduction} \label{introsec}
Let $H(\D)$ denote the collection of  holomorphic functions on the open unit disk $\D$ in the complex plane $\CP$.  Let $B$ be a Banach space of functions contained in $H(\D)$ such that for each $z\in \D$ the linear functional  $f\mapsto f(z)$ is continuous on $B$.   Let $\psi\in H(\D)$ be a {\it  multiplier} of $B$, meaning that $\psi g\in B$ whenever $g\in B$.  When $\psi$ is a multiplier of $B$, we define the multiplication operator $M_\psi: B \to B$ by $M_\psi g = \psi g$ and call $\psi$ the symbol of $M_\psi$.  Note that $M_\psi$ is continuous on $B$ by the Closed-Graph Theorem.  When $\psi$ is the identity function, $\psi(z) = z$, we denote $M_\psi$ by $M_z$ and, more generally, we sometimes use the notation $M_{\psi(z)}$ for $M_\psi$.

Recall that a Fredholm operator on a Banach space $B$ is a bounded linear operator having a finite-dimensional kernel and a closed, finite-codimensional range. (We remark that finite codimensionality of the range implies the range is closed; see, e.g., \cite[Lemma 1.2.7]{MMM}.) In the late 1960s, R.G.\ Douglas characterized Fredholm multiplication operators on the Hardy Hilbert space $H^2$ of $\D$ as a consequence of his work on Toeplitz operators (see, e.g., \cite[Corollary, p.\ 897]{RDB} or \cite[Theorem 7.36]{DBAT}).   The multipliers of the Hardy space $H^2(\D)$ are precisely the bounded analytic functions on $\D$, i.e., the functions in $H^\infty(\D)$.  For $\psi \in H^\infty(\D)$, Douglas showed that $M_\psi: H^2(\D) \to H^2(\D)$ is Fredholm if and only if $\psi$ is bounded away from $0$ near $\partial \D$, meaning there is an $\epsilon > 0$ and a number $r$ satisfying $0 < r < 1$ such that $|\psi(z)| > \epsilon$ whenever $r < |z| < 1$.  

For multiplication operators on the Bergman Hilbert space $A^2(\B_n)$, where $\B_n$ is the open unit ball of $\CP^n$, McDonald \cite[Theorem 1.2]{McD} showed that Douglas's characterization continues to hold:  a multiplication operator on $A^2(\B_n)$ is Fredholm if and only if its symbol is bounded away from 0 near $\partial \B_n$.   Using Banach-space techniques, Axler \cite[Theorem 23]{AFC} showed that not only does this characterization hold for the Bergman spaces $A^p(\D)$, $p\ge 1$, but it also holds for  $A^p(W)$ for any open subset $W$ of $\CP$ provided $\partial W$ is replaced by  $\partial_{p\text{-}e}W$,  the``$p$-essential boundary'' of $W$.      Cao, He, and Zhu \cite[Theorem 15]{CHZ}  extended   Douglas's characterization to the Hardy-Sobolev spaces $H^2_\beta(\B_n)$, $\beta \in \R$, and this extension is the primary source of motivation for our work. 

    As Cao, He, and Zhu demonstrate \cite[Lemmas 9 and 13]{CHZ}, the problem of characterizing Fredholm multiplication operators on spaces of analytic functions  on $\B_n$ for $n > 1$ is considerably simpler than that for $n=1$ because zeros of analytic functions of several complex variables are not isolated.   Thus, we choose to focus on spaces of analytic functions  on $\B_1 = \D$.  Our goal is extend Douglas's characterization of Fredholm multiplication operators to a Banach-space setting where the spaces are defined by seven properties satisfied by all the Hardy-Sobolev spaces $H^2_\beta(\D)$,  as well as, e.g.,  the Hardy spaces $H^p(\D)$, weighted Bergman spaces $A^p_\omega(\D)$, and the derivative Hardy spaces $S_j^p(\D)$  (where $p\ge 1$ for all these examples, $\beta \in \R$, $\omega > -1$, and $j\in \Z_{>0}$; see Section~\ref{examples} for definitions of these spaces).  The first three of the seven defining properties for the Banach spaces $\BS$ to which we extend Douglas's characterization of Fredholm multiplication operators are (i) density of the polynomials, (ii) continuity of point-evaluation functionals for points in $\D$, and (iii)  $M_z: \BS\to \BS$ is bounded and has spectral radius $1$.  All the properties defining $\BS$ are described in the next section (Definition~\ref{defBS}). 
    
   Suppose that $B$ is a Banach space of  analytic functions on $\D$ such that $M_\psi$ and $M_z$ are multiplication operators on $B$. Observe that if $M_\psi$ is Fredholm, then  because $M_\psi$ and $M_z$ commute, $\Ran(M_\psi)$ is a closed, finite-codimensional invariant subspace for $M_z$.     A secondary source of motivation for our work is provided by results obtained by Axler and the first author  in \cite{AB} on  closed, finite-codimensional, multiplication by $z$ invariant subspaces of Bergman spaces of planar domains having boundary with no connected component equalling a point.   Because $\D$ is such a planar domain,  Theorem 5 and Corollary 6 of \cite{AB} show that for the Bergman spaces $A^p(\D)$, $p\ge 1$, every closed, finite-codimensional subspace invariant under $M_z$ is the range of a Fredholm multiplication operator.  In fact, Corollary 6 of \cite{AB} shows how a characterization of the closed, finite-codimensional, $M_z$-invariant subspaces of $A^p(\D)$ yields, as a corollary, that  Fredholm multiplication operators on $A^p(\D)$ have symbols that must be bounded away from $0$ near $\partial \D$. 
   
     For  Banach spaces $\BS$ to which we extend Douglas's characterization of Fredholm multiplication operators, we also characterize the closed, finite-codimensional, $M_z$-invariant subspaces, showing that these subspaces are precisely the ranges of Fredholm multiplication operators for some spaces but not for others---for the others, every closed, finite-codimensional, $M_z$-invariant subspace is either the range of a Fredholm multiplication operator (whose symbol is necessarily bounded away from $0$ near $\partial \D$) or the closure of the range of a multiplication operator whose symbol is not bounded away from $0$ near $\partial \D$.   These $M_z$-invariant subspaces are determined by the common zeros of functions within them  and are singly generated: there is a polynomial $q$ whose degree is the subspace's codimension such that the subspace is the closure of $q\BS$, so that $q$ is a cyclic vector for $M_z$ restricted to the subspace. 

   We remark that Beurling's Theorem (\cite{ABeur}) describing all closed, $M_z$-invariant subspaces of $H^2(\D)$ (in particular the finite codimensional ones) yields, as a corollary, that a multiplication operator on $H^2(\D)$ is Fredholm if and only if its symbol is a finite Blaschke product times an outer function $f$ such that both $f$ and $1/f$ belong to in $H^\infty(\D)$. Such symbols describe the multipliers of $H^2(\D)$ that are bounded away from $0$ near $\partial\D$.     We have included this remark because during the period between the publication of Beurling's Theorem in 1949 and Douglas's work on Toeplitz operators in the late 1960s, some mathematicians must have noted how Beurling's Theorem leads to Douglas's characterization of Fredholm multiplication operators on $H^2(\D)$; however,  we have been unable to find a reference from this period that uses Beurling's Theorem to characterize Fredholm multiplication operators on $H^2(\D)$.

 \section{Preliminaries}\label{PLS}

 We denote the closure of a set $S$ by $S^-$.  For each nonnegative integer $k$, let  $C^{k}(\D^-)$ denote the Banach space consisting of functions that are continuously differentiable on the closed disk $\D^-$ through order $k$, with $\|f\|_{C^k(\D^-)} = \sum_{j=0}^k \left\|f^{(j)}\right\|_\infty$.  For some of the Banach spaces of analytic functions on $\D$ that we consider, every function in the space extends to be in $C^{k}(\D^-)$ for some $k\in \Z_{\ge 0}$.   In the definition below,  we introduce the notion of  boundary-extension type  $n$, where $n$ is a nonnegative integer.  As you might guess, for spaces of boundary-extension type $0$, some functions in the space do not extend continuously to $\partial \D$ (see Remark (iv) following the definition).   

\begin{definition}\label{BETD} Let $B$ be a Banach space of functions in $H(\D)$ such that the operator of multiplication by $z$,  $M_z: B \to B$, is bounded, and such that for each $z\in \D$, the point-evaluation functional $E_z:B \to \CP$, defined by $E_z f = f(z)$, is continuous.  We  say that $B$ has {\it boundary-extension type $n$} for some positive integer $n$ provided that 
\begin{itemize} 
\item[(n.1)] $B\subset C^{n-1}(\D^-)$, and
\item[(n.2)]  for each $\zeta\in \partial \D$, the closure of the range of $M_{(z-\zeta)^{n+1}}: B\to B$ contains $f(z) = (z - \zeta)^n$.
\end{itemize}
We say that  $B$ has {\it boundary-extension type $0$} provided that 
\begin{itemize}
\item[(0.1)]  for each $\zeta\in \partial \D$, the closure of the range of $M_{z-\zeta}: B\to B$ contains $f(z) = 1$.
\end{itemize}
\end{definition}
Remarks: 
\begin{itemize}
\item[(i)] For a positive integer $n$,  $B\subset C^{n-1}(\D^-)$ means that each $f\in B$ extends to $\partial \D$ and this extension, which we will also denote by $f$, belongs to $C^{n-1}(\D^-)$.  
\item[(ii)] By the Closed-Graph Theorem, if $B\subset C^{n-1}(\D^-)$, then the identity operator from $B$ into $C^{n-1}(\D^-)$ is continuous, so that $\|f\|_{C^{n-1}(\D^-)} \le \gamma\|f\|_B$ for some positive constant $\gamma$ and all $f \in B$.  {\em Thus, if $B$ has boundary-extension type $n$ for a positive integer $n$ and $\zeta\in \partial \D$, then for every nonnegative integer $j < n$, the linear functional $E_{\zeta}^{(j)}: B\to \CP$ defined by $E_\zeta^{(j)}f = f^{(j)}(\zeta)$ is continuous.}  
\item[(iii)]  If $B$ has boundary-extension type $k$ for some positive integer $k$, then it does not have boundary-extension type $k-1$ by the following argument:  Suppose, in order to obtain a contradiction, that $B$ has  boundary-extension type $k\in \Z_{>0}$ as well as boundary-extension type $k-1$.  Because $B$ has boundary-extension type $k$, as noted in (ii) above, the linear functional  $E^{(k-1)}_\zeta$ is continuous on $B$ for each $\zeta\in \partial \D$.   Because $B$ has boundary-extension type $k-1$, for each $\zeta\in \partial \D$, there is  some sequence $(h_j)$ in $B$ such that $(M_{(z-\zeta)^{k}}h_j)$ converges in $B$ to  $f(z) = (z-\zeta)^{k-1}$, so that by the continuity of $E_{\zeta}^{(k-1)}$ on $B$, we  have  $E_{\zeta}^{(k-1)}((z-\zeta)^kh_j(z))\to (k-1)!$, implying $0=(k-1)!$, a contradiction.  
\item[(iv)] If $B$ has boundary-extension type $0$ and $\zeta\in \partial \D$, then the argument of (iii) above, with $k=1$, shows $E_\zeta$ cannot be continuous on $\BS$.  Thus, $B \not\subset C^{0}(\D^-)$.\\\end{itemize}

We now present examples of Banach spaces of analytic functions on $\D$ having boundary-extension type $0$ as well as examples having {\it  positive boundary-extension type}, which means ``having boundary-extension type $n$ for some positive integer $n$.'' For verifications that the spaces described below have the claimed extension types, see  Section~\ref{examples} below.
For $p\ge 1$, the Hardy and Bergman spaces, $H^p(\D)$ and $A^p_\omega(\D)$, respectively,  are examples of spaces having boundary-extension type $0$.  For $p\ge 1$, the spaces $S^p(\D)$ consisting of functions $f$ such that $f' \in H^p(\D)$, with $\|f\|^p_{S^p(\D)} =  |f(0)|^p + \|f'\|^p_{H^p(\D)}$ are spaces having boundary-extension type $1$.  More generally, for $p\ge 1$ and $j\in \Z_{> 0}$, the spaces $S^p_j(\D)$  consisting of functions having $j$-th derivative in $H^p(\D)$ have boundary-extension type $j$. (Note $S^p_1(\D) = S^p(\D)$.) The Hardy-Sobolev spaces $H^2_{\beta}(\D)$ (described in \cite{CHZ})  have boundary-extension type $0$ if $\beta \le 1/2$ and boundary-extension type $k\in \Z_{>0}$ if $k-1/2<\beta\le k+1/2$.  For $\beta =0$, $H^2_{\beta}(\D)$ is the Hilbert space $H^2(\D)$; for $\beta < 0$, $H^2_{\beta}(\D)$ corresponds to the weighted Bergman space $ A^2_{-1-2\beta}(\D)$; and,  for $0 < \beta \le 1/2$, $H^2_\beta(\D)$ corresponds to the weighted Dirichlet space consisting of functions having derivative in $A^2_{1 - 2\beta}(\D)$. When $\beta = k$ for some positive integer $k$, $H^2_\beta(\D)$ corresponds to $S^2_{\beta}(\D)$.

\begin{definition}\label{defBS} Let $\BS$ denote a Banach space of analytic functions on $\D$ having the following properties:
\begin{itemize}
\item[(i)]  $\BS$ contains the polynomials as a dense subspace.
\item[(ii)]  For each $z\in \D$,  the point-evaluation functional $E_z: \BS \to \CP$ is continuous.
\item[(iii)] The operator $M_z: \BS\to \BS$ is bounded and has spectral radius $1$.
\item[(iv)] If $q$ is a polynomial all of whose zeros lie in $\D$ and $qf\in \BS$ for some function $f\in H(\D)$, then $f\in \BS$.
\item[(v)] If $g$ is a multiplier of  $\BS$ that is bounded away from $0$ on $\D$, then $1/g$ is also a multiplier of $\BS$.  
\item[(vi)]  Let $l$ be a positive integer.  If $(q_k)$ is a sequence of polynomials each of degree at most $l$ and  $q_k(z) = \sum_{j=0}^l a_{j,k}z^j$ is such that for each $j\in\{0,1, \ldots, l\}$  the coefficient sequence $(a_{j,k})_{k=1}^\infty$ converges to $0$, then $(q_k)$ converges to $0$ in $\BS$.
\item[(vii)] $\BS$ has boundary-extension type $n$ for some nonnegative integer $n$.
\end{itemize}
\end{definition}

Observations.
\begin{enumerate}
\item  Via the Uniform Boundedness Principle, property (ii) implies that for each nonnegative integer $j$ and $z\in \D$, the linear functional $E_z^{(j)}: \BS\to \CP$ defined by $E_z^{(j)} f = f^{(j)}(z)$ is continuous. We sketch the proof of the preceding implication.   Suppose that $E_z$ is continuous for each $z\in \D$ and that $K$ is a compact subset of $\D$.  Then the collection of continuous linear functionals $\{E_z: z\in K\}$ is uniformly bounded in norm. Thus, if $(f_n)$ is a sequence in $\BS$ converging to $f$ in $\BS$, then $(f_n)$ converges uniformly on compact subsets of $\D$ to $f$, from which the continuity of the linear functionals $E_z^{(j)}$, for $z\in \D$ and $j\in \Z_{\ge 0}$ follows.  

\item Because $\BS$ contains all polynomials, it is infinite dimensional; also, it follows that any collection of linear functionals of the form $E^{(j)}_z$, where $z\in \D$ and $j\ge 0$ will be linearly independent as long as no two functionals are the same. Similarly, if $\BS$ has boundary-extension type $n$ for some positive integer $n$, then any collection of linear functionals whose members have the form $E^{(j)}_z$ for $z\in \D$ and $j\ge 0$ or $E_\zeta^{(j)}$ for $\zeta\in \partial \D$  and $0\le j\le n-1$ will be linearly independent as long as no two functionals are the same.

\item  Because the point-evaluation functionals $E_z$, $z\in \D$, are bounded on $\BS$, it is straightforward to apply the Closed-Graph Theorem to see that any multiplier of $\BS$ must induce a bounded multiplication operator on $\BS$. 
\item Let $\psi$ be a multiplier of $\BS$. Note that for every $a\in \D$, $M_\psi - \psi(a) I = M_{\psi - \psi(a)}$ cannot be invertible on $\BS$ because every function in its range vanishes at $a$, and there are functions in $\BS$ that do not vanish at $a$.  In other words, $\psi(a)$ belongs to the spectrum of $M_\psi$ for every $a\in \D$.  It follows that $|\psi(a)|\le \|M_\psi\|$ for every $a\in \D$.  {\em Thus, every multiplier of $\BS$ is a bounded function on $\D$.} 
\item It is easy to see that any multiplication operator on $\BS$ with nonzero symbol must be injective; thus, $M_\psi: \BS \to \BS$ is Fredholm if and only if the range of $M_\psi$ is closed and finite codimensional.   
\item Note that property (iii) implies that every polynomial is a multiplier of $\BS$. 
\item Property (v) is shown to hold for the Hardy-Sobolev spaces $H^2_\beta(\B_n)$ in Proposition 6 of \cite{CHZ}, which is described as a key result for determining the essential spectrum of a multiplication operator on $H^2_\beta(\B_n)$ (from which Douglas's characterization of Fredholm multiplication operators follows). 
\item \label{MOR} Suppose that a multiplier $h$ of $\BS$ is contained in a closed $M_z$-invariant subspace $\N$ of $\BS$; then the closure of the range of $M_h: \BS \to \BS$ is contained in $\N$.  (If $h\in \N$, then so is $hp$ for every polynomial $p$. Because the polynomials are dense in $\BS$ and $\N$ is closed, it follows that $h\BS\subseteq \N$, and thus the closure of the range of $M_h$ is contained in $\N$.)
\item  \label{p6a}  As a special case of the preceding observation and  property (i), we see that if $\BS$ has boundary-extension type $0$, then  $\Ran(M_{z-\zeta})$ is dense for every $\zeta\in \partial \D$. The converse is obvious;  thus, {\em $\BS$ has boundary-extension type $0$ if and only if  $M_{z-\zeta}: \BS\to \BS$ has dense range for all $\zeta\in \partial \D$}. 
\end{enumerate}

In Section~4, we show that the spaces of functions $H^p(\D)$, $A^p_\omega(\D)$, $H^2_\beta(\D)$, and $S^p_j(\D)$ (where $p\ge1$, $\omega > -1$, $\beta\in \R$, and $j\in \Z_{>0}$) are all examples of spaces $\BS$ satisfying  properties (i)--(vii) of Definition~\ref{defBS}. In the next section, we characterize the Fredholm multiplication operators on the spaces $\BS$, as well as the closed, finite-codimensional, $M_z$-invariant subspaces of $\BS$, and we relate the two characterizations.  

\section{Characterizing the Fredholm multiplication operators on $\BS$}

Our principal result is the following, where $\BS$ represents any Banach space satisfying the properties of Definition~\ref{defBS}.  

\begin{theorem}\label{TMT} The multiplication operator $M_\psi: \BS\to \BS$ is Fredholm if and only if $\psi$ is bounded away from $0$ near $\partial \D$.
\end{theorem}

Our proof of the preceding theorem depends on a number of lemmas and propositions contributing to proofs of each direction of the theorem, with the ``if direction'' corresponding to Theorem~\ref{TED}  and the ``only-if direction'' corresponding to Theorem~\ref{THD}.  

The following lemma generalizes Proposition 1 of \cite{AB}; the proof is essentially the same.   

\begin{lemma}\label{MqL} Suppose that $\BS$ is a Banach space of analytic functions on $\D$  satisfying Definition~\ref{defBS} and that $q$ is a nonzero polynomial whose zeros lie in $\D$. Then $q\BS$ is a closed, finite-codimensional subspace of $\BS$ invariant under $M_z$; moreover, $\dim\left(\rule{0in}{0.14in}\BS/(q\BS)\right) = \deg(q).$  \end{lemma}
\begin{proof}    If $\deg(q) = 0$, then $q$ is a nonzero-constant polynomial and $q\BS = \BS$ is closed, has codimension $0$, and is $M_z$ invariant by property (iii) of $\BS$.     Suppose $\deg(q) > 0$ and $q$ has precisely $l$ different zeros in $\D$. Let 
$$
Z =\{(w_1, m_1), (w_2, m_2), \ldots, (w_\l, m_\l)\}
$$ 
be such that $w_1, w_2, \ldots, w_\l$ are the distinct zeros of $q$ with $m_i$ being the multiplicity of $w_i$ for $1\le i \le l$.  (Thus, $\deg(q) = \sum_{i=1}^l m_i$.)  Let $E_q = \cup_{i=1}^l \left(\cup_{j=0}^{m_i-1}\{E^{(j)}_{w_i}\}\right)$ be the set of $\deg(q)$ linearly independent linear functionals corresponding to the zeros of $q$ (and their multiplicities).  Observe that any element of $q\BS$ belongs to the intersection of the kernels of the linear functionals in $E_q$.  Conversely, by property (iv) of $\BS$, if $g$ belongs to the intersection of the kernels of the linear functionals in $E_q$, then $g \in q\BS$.  Hence, $q\BS$ is the intersection of the $\deg(q)$ linearly independent continuous linear functionals in $E_q$; thus, $q\BS$ is closed and has codimension $\deg(q)$ in $\BS$.\end{proof} 

Observe that $\BS$ must satisfy only (ii)-(iv) of Definition~\ref{defBS} for the preceding lemma to hold (and that for (iii), the spectral radius assumption is not needed).

One direction of the proof of  Theorem~\ref{TMT} is straightforward.  

\begin{theorem} \label{TED} Suppose that $\BS$ is a Banach space of analytic functions on $\D$  satisfying Definition~\ref{defBS}. If $\psi$ is a multiplier of $\BS$ that is bounded away from $0$ near $\partial\D$, then $M_\psi: \BS\to \BS$ is Fredholm.
\end{theorem}  
\begin{proof} Suppose that $\psi$ is a multiplier of $\BS$ that is bounded away from zero near $\partial \D$. Then $\psi = qg$, where $q$ is a polynomial whose zeros lie in $\D$, and $g\in H(\D)$ is both bounded on $\D$ and bounded away from $0$ on $\D$.  Note that $g$ is bounded  on $\D$ because $\psi$, being a multiplier of $\BS$, is bounded on $\D$; also note that $g\in \BS$ by property (iv) of $\BS$.  

Observe that any function of the range of $M_\psi$ must vanish at the zeros of $q$ (according to their multiplicities); thus, any function in the range of $M_\psi$ factors as $qh$ for some $h\in H(\D)$ and it follows that $h\in\BS$ by property (iv) of $\BS$.  We see $\Ran(M_\psi) \subseteq q\BS$.

 Let $f\in \BS$ be arbitrary.  Because $\psi$ is a multiplier of $\BS$, we have $\psi f = qg f\in \BS$.  Because $q(gf)\in \BS$, property (iv) of $\BS$ implies $gf\in \BS$.  Because $f$ is arbitrary, we see that $gf\in \BS$ for all $f\in \BS$, so that $g$ is a multiplier of $\BS$.  Because $g$ is bounded away from $0$ on $\D$,  $1/g$ is also a multiplier of $\BS$ (property (v)). Thus, for an arbitrary $f\in \BS$, we have $\frac{1}{g}f\in \BS$, and
$$
qf = q g \frac{1}{g} f = \psi \frac{1}{g} f = M_\psi \left( \frac{1}{g}f\right).
$$
Thus, $q\BS\subseteq \Ran\left(M_\psi\right)$, and, because the reverse inclusion also holds, we have $\Ran(M_\psi) = q\BS,$ a closed subspace of $\BS$ having finite codimension by Lemma~\ref{MqL}.  Because $M_\psi$ is injective, we see that $M_\psi$ is Fredholm, which completes the proof.    
\end{proof} 

Observe that $\BS$ must satisfy only (ii)-(v) of Definition~\ref{defBS} for the preceding theorem to hold (and that for (iii), the spectral radius assumption is not needed).  The other direction of our proof of Theorem~\ref{TMT} depends upon several lemmas as well as a theorem characterizing the closed, finite-codimensional, $M_z$-invariant subspaces of $\BS$.

The following lemma is based on an idea of Gellar \cite[Section 9]{Ge}. 

\begin{lemma}\label{CAP} Suppose that $\BS$ is a Banach space of analytic functions on $\D$ satisfying Definition~\ref{defBS} above.  If $\N\subseteq \BS$ is a closed, finite-codimensional subspace of $\BS$, and $\N$ is invariant under $M_z$, then $\mathcal{N}$ contains a monic polynomial of degree at most $\dim(\BS/\N)$.\end{lemma}
\begin{proof}  Let  $\N\subseteq \BS$ be a closed, finite-codimensional subspace of $\BS$ that is invariant under $M_z$.  If $\N = \BS$, then $\N$ contains the monic polynomial $p\equiv 1$ by property (i) of $\BS$.  For the remainder of the proof, we assume $\N$ has positive codimension in $\BS$.  

  Define a linear mapping  $T:\BS/\N \to \BS/\N$ by $T(f + \N) = zf + \N$, noting that $T$ is well defined by the invariance of $\N$ under $M_z$.  Because $\BS/\N$ is a finite dimensional vector space, $T$ has a minimal polynomial $p$ having degree at most $\dim(\BS/\N)$ such that $p(T) = 0$; equivalently, $p(z)\BS \subseteq \N$. Because $\BS$ contains the constant function $z\mapsto 1$, we see $p$ is a polynomial of degree at most $\dim(\BS/\N)$ that is contained in $\N$, and, being a minimal polynomial, $p$ is monic. \end{proof}

The following  lemma generalizes Theorem 2 of \cite{AB}, and its proof is essentially the same as that in \cite{AB}.

\begin{lemma}\label{FCIV} Suppose that $\BS$ is a Banach space of analytic functions on $\D$ satisfying Definition~\ref{defBS} and that $\BS$ has boundary-extension type $0$.  If $\N\subseteq \BS$ is a closed, finite-codimensional subspace of $\BS$ that is invariant under $M_z$, then there is a monic polynomial $q$ whose zeros lie in $\D$ such that $\N = q\BS$.  Moreover, $\deg(q) = \dim(\BS/\N)$, so that the number of zeros of $q$ in $\D$ counting multiplicities is  $\dim(\BS/\N)$.
\end{lemma}
\begin{proof}  Let  $\N\subseteq \BS$ be a closed, finite-codimensional subspace of $\BS$ that is invariant under $M_z$.  If $\N = \BS$, choose $q\equiv 1$. For the remainder of the proof, we assume $\N$ has positive codimension in $\BS$.   

From Lemma~\ref{CAP} above, we know there is a monic polynomial $p$ with $\deg(p) \le \dim(\BS/\N)$ such that $p\in \N$.
Let $p = qh$, where $q$ is a  monic polynomial whose zeros lie in $\D$ and $h$ is a monic polynomial whose zeros (if any) have modulus $\ge 1$.  If $\lambda$ is a zero of $h$ such that $|\lambda|  >1$, then $M_{z-\lambda}$ is invertible by our assumption that $M_z: \BS\to \BS$ has spectral radius 1. If $\lambda$ is a zero of $h$ such that $|\lambda|  =1$, then $M_{z-\lambda}$ has dense range because we are assuming $\BS$ has boundary-extension type $0$ (see Observation~\ref{p6a} following the statement of Definition~\ref{defBS}).   It follows that $M_h: \BS \to \BS$ has dense range because either $h\equiv 1$ or is a product of dense-range multiplication operators of the form $M_{z-\lambda}$, with $\lambda$ being a zero of $h$.   Thus, $q\BS = (p \BS)^- \subseteq \N$.   We have
\begin{align*}
\dim (\BS/\N) & \le \dim (\BS/(q\BS)) \quad (q\BS\subseteq \N)\\
 &= \deg(q) \quad (\text{Lemma}~\ref{MqL})\\
 & \le \deg(p) \quad  ( \text{definitions of}\ p\ \text{and}\ q)\\
 & \le \dim(\BS/\N) \quad ( \text{definition of}\ p).
\end{align*}
It follows that  $\dim (\BS/\N) = \deg(q) = \dim (\BS/(q\BS))$, and, because $q\BS\subseteq \N$, we must have $\N= q\BS$, as desired.
\end{proof}

\begin{lemma}  \label{AZOT}  Suppose that $\BS$ is a Banach space of analytic functions on $\D$ satisfying Definition~\ref{defBS} and that $\BS$ has boundary-extension type $n$ for some positive integer $n$. Suppose $p$ is a polynomial of positive degree. Then there is a polynomial $q$ with $\deg(q)\le \deg(p)$ such that the following hold.\par
(a) $q$ has no zeros outside of $\D^-$.\par
(b) All zeros of $q$ in $\D$ are precisely the zeros of $p$ in $\D$ with the same multiplicities.  \par
(c) All zeros of $q$ in $\partial\D$ are precisely the zeros of $p$ in $\partial \D$, but each has multiplicity at most $n$.\par
(d) $q$ is in the closure of the range of $M_p:\BS \to \BS$.
\end{lemma}

\begin{proof}  Suppose that $p$ has some zeros not in $\D^-$. Then we can factor $p$ as $p = \tilde{p} h$ where $\tilde{p}$ is a polynomial whose zeros are precisely the zeros of $p$ in $\D^-$ (including multiplicities), and $h$ is a polynomial whose zeros (including multiplicities) are precisely those of $p$ outside of  $\D^-$.  Because the spectral radius of $M_z:\BS\to \BS$ is 1, the operator $M_h: \BS\to \BS$ is invertible, being a nonzero constant times a product of invertible operators of the form $M_{z-\lambda}$ where $\lambda$ is a zero of $h$ (necessarily satisfying $|\lambda| > 1$).  Thus, $\Ran(M_{p}) = \Ran(M_{\tilde{p}h})= \Ran(M_{\tilde{p}})$. If $p$ has all its zeros in $\D^-$, then set $\tilde{p}=p$.  We have shown  that $\Ran(M_{p}) = \Ran(M_{\tilde{p}})$, where $\tilde{p}$ is a polynomial having no zeros in $\CP\setminus\D^-$ and whose zeros in $\D^-$ are precisely those of $p$ in $\D^-$ (including multiplicities).

Suppose that $\tilde{p}$ has precisely $d >0$ different zeros in $\partial \D$. Because $\tilde{p}$ has no zeros outside of $\overline{\D}$,  $\tilde{p}$ factors as  $\tilde{p}(z) =  g(z)\prod_{i=1}^d (z-\zeta_i)^{\mu_i}$, where $g$ is a polynomial whose zeros are precisely the zeros of $\tilde{p}$ in $\D$ (including multiplicities), $\zeta_1, \ldots, \zeta_d$ are the $d$ different zeros of $\tilde{p}$ in $\partial \D$, and, for $1\le i \le d$,  $\mu_i$ is the multiplicity of $\zeta_i$.     If $\tilde{p}$ has no zeros in $\partial \D$ with multiplicity greater than $n$, then set $q=\tilde{p}$ (and observe $q$ satisfies (a)--(d) of the lemma).

 Suppose that for some integer $i_0$, $1 \le i_0 \le d$, $\zeta_{i_0}$ is a zero of $\tilde{p}$ with multiplicity $n+1$ or greater; i.e., $\mu_{i_0} \ge n+1$.  Then $\tilde{p}(z)= u(z)(z - \zeta_{i_0})^{n+1}$, where $u$ is the polynomial given by 
 $$
 u(z) = g(z) \big(z-\zeta_{i_0}\big)^{\mu_{i_0} -n-1} \prod_{\myov{i=1}{i \neq i_0}}^d \big(z-\zeta_i\big)^{\mu_i}.
 $$
 Because $\BS$ has boundary-extension type $n$, $M_{(z-\zeta_{i_0})^{n+1}}$ has $f(z) = (z-\zeta_{i_0})^n$ in the closure of its range, and, by the continuity of $M_u:\BS\to \BS$, we see $M_{\tilde{p}} = M_uM_{(z-\zeta_{i_0})^{n+1}}$ has $uf$ in the closure of its range; thus,  $\left(\Ran(M_{\tilde{p}})\right)^-$ includes 
$q_1(z) = u(z)(z-\zeta_{i_0})^n = \frac{\tilde{p}(z)}{z-\zeta_{i_0}}$. 

 If $\mu_{i_0} =n+1$, then $q_1$ has $\zeta_{i_0}$ as a zero of multiplicity $n$. If $\mu_{i_0} > n+1$, then because $q_1$ is in the closure of the range of $M_{\tilde{p}}$ and has $\zeta_{i_0}$ as a zero of multiplicity at least $n+1$, we may repeat the argument of the preceding paragraph to obtain that $q_2(z) = \frac{\tilde{p}(z)}{(z-\zeta_{i_0})^2}$ belongs to the closure of the range of $M_{q_1}$, and hence to the closure of the range of $M_{\tilde{p}}$ by Observation~\ref{MOR} of the list following Definition~\ref{defBS}.   Continuing this process if necessary, we obtain a polynomial $q_{\mu_{i_0}-n}$ in the closure of the range of $M_{\tilde{p}}$ having $\zeta_{i_0}$ as  a zero of multiplicity $n$.  Note that the zeros of $q_{\mu_{i_0}-n}$ are precisely those of $\tilde{p}$ with the same multiplicities except for that of $\zeta_{i_0}$.
 
Suppose  $\tilde{p}$ has a zero different from $\zeta_{i_0}$, say $\zeta_{i_1}$, where $i_1\in \{1,...,d\} \setminus \{i_0\}$,  such that $\zeta_{i_1}$ also has multiplicity exceeding $n$.  Therefore, $\zeta_{i_1}$ is a zero of multiplicity greater than $n$ for $q_{\mu_{i_0}-n}$ as well, and by what we have established above, the closure of the range of $M_{q_{\mu_{i_0}-n}}$ includes a polynomial necessarily in the closure of the range of $M_{\tilde{p}}$ whose zeros are precisely those of $\tilde{p}$ with the same multiplicities  except that both $\zeta_{i_0}$ and $\zeta_{i_1}$ have multiplicity $n$.   Continuing this process if necessary, we see that the closure of the range of $M_{\tilde{p}}$ contains a polynomial $q$ whose zeros are precisely those of $\tilde{p}$, with the same multiplicities for zeros in $\D$ and with multiplicities at most $n$ for zeros in $\partial \D$.    Because $\tilde{p}$ has no zeros in $\CP\setminus\D^-$, the zeros of $\tilde{p}$ in $\D^-$ are precisely those of $p$ in $\D^-$, and $\Ran(M_p)=\Ran(M_{\tilde{p}}) $, the lemma follows.
\end{proof}

 We now introduce notation for closed, finite-codimensional subspaces of $\BS$ determined by zeros. Rather than working with lists of zeros repeated according to multiplicities, we will work with zero-specification sets, which will consist of ordered pairs with the first entry being a complex number specifying a zero and the second entry being a positive integer specifying a minimum multiplicity. 

  If $\BS$ satisfies Definition~\ref{defBS} and has boundary-extension type $0$, then a  {\it finite zero-specification set} $Z$ for $\BS$ takes the form
  \begin{equation}\label{ZsetZ} 
  Z= \cup_{i=1}^l\left\{\rule{0in}{.13in}(w_i, m_i)\right\},
  \end{equation}
 where $l$ is a nonnegative integer, $\{w_i: 1 \le i \le l\}$ consists of $l$ different numbers in $\D$,  and $m_i$ is a positive integer for $1\le i \le l$.  As usual, we take a union indexed from integers $1$ to $l$ to be empty if $l <1$; thus, zero-specification sets can be empty.  If $\BS$ satisfies Definition~\ref{defBS} and has  boundary-extension type $n$ for some positive integer $n$,  then a {\it finite zero-specification set} $Z$ for $\BS$  takes the form 
 \begin{equation}\label{ZsetP}
Z = \cup_{i=1}^l\left\{\rule{0in}{.13in}(w_i, m_i)\right\}  \bigcup \cup_{i=1}^d\left\{\rule{0in}{.13in}(\zeta_i,\mu_i)\right\},
\end{equation}
 where $l$ and $d$ are nonnegative integers such that  $\{w_i: 1\le i \le l\}$ consists of $l$ different numbers in $\D$,  $\{\zeta_i:  1\le i\le d\}$ consists of $d$ different numbers in $\partial\D$,  $m_i$ is a positive integer for $1\le i \le l$, and $\mu_i\le n$ is a positive integer for $1\le i\le d$, where $n$ is the boundary-extension type for $\BS$.   

  \begin{definition}\label{dzs}   For a Banach space $\BS$ of analytic functions on $\D$ satisfying Definition~\ref{defBS}, we define the closed, finite codimensional subspace $S_Z$ of $\BS$ determined by a finite zero-specification set $Z$ as follows.  
  \begin{itemize}
\item[(i)]  If $Z = \emptyset$, then $S_Z = \BS$. 
\item[(ii)]  If $\BS$ has boundary-extension type $0$ and  $Z$, given by (\ref{ZsetZ}), is non-empty, then 
  \begin{align}\label{ZSD}
  S_Z  =  \text {the intersection of the kernels of the linear functionals on}\ \BS\  \text{belonging to} \\   E_{\D}:=  \cup_{i=1}^l \left(\cup_{j=0}^{m_i-1}\left\{E_{w_i}^{(j)}\right\}\right). \rule{2in}{0in}\nonumber
\end{align}
 \item[(iii)] If $\BS$ has boundary-extension type $n$ for some positive integer $n$ and $Z$, given by (\ref{ZsetP}), is non-empty, then
 \begin{align}\label{ZSDB}
  S_Z  =  \text {the intersection of the kernels of the linear functionals on}\ \BS\  \text{belonging to}\\  E_{\D^-}:=  \cup_{i=1}^l \left(\cup_{j=0}^{m_i-1}\left\{E_{w_i}^{(j)}\right\}\right) \bigcup\   \cup_{i=1}^d  \left(\cup_{j=0}^{\mu_i-1}\left\{E_{\zeta_i}^{(j)}\right\}\right). \rule{1in}{0in}\nonumber
\end{align}
\end{itemize}
   \end{definition}
  
  Remarks: (a) If, for example, $Z$ is given by (\ref{ZsetP}), then 
  $$
  S_Z = \left\{f\in \BS:\   \begin{array}{l} \text{for} \ 1\le i \le l\,\  f^{(j)}(w_i) = 0 \ \text{for}\ 0\le j\le m_i -1;\  \text{and}\\ \text{for}\ 1 \le i \le d, f^{(j)}(\zeta_i) = 0\ \text{for}\ 0\le j\le \mu_i - 1\end{array} \right\},  
 $$ 
 so that $S_Z$ consists of those functions in $\BS$ whose zeros and corresponding minimum multiplicities are specified by $Z$.

  (b) Suppose that $\BS$ has boundary-extension type $n$ for some positive integer $n$.  Because functions in $\BS$ need not have derivatives  at points of $\partial \D$, relative to $\D^-$, of order exceeding $n-1$, we say that function $f\in\BS$ has $\zeta\in \partial \D$ as a zero of multiplicity $n$ provided $f^{(k)}(\zeta)= 0$ for integers $k$ satisfying $0\le k\le n-1$ and either $f^{(n)}(\zeta) \ne 0$ or $f^{(n)}(\zeta)$ does not exist.
  
   (c) Because the linear functionals of $E_{\D}$ of (\ref{ZSD}) are linearly independent, the closed subspace $S_Z$ of $(\ref{ZSD})$ has codimension $\sum_{i=1}^l m_i$ in $\BS$,  and, similarly, the closed subspace $S_Z$ of (\ref{ZSDB}) has  codimension $\sum_{i=1}^l m_i + \sum_{i=1}^d \mu_i$ in $\BS$. 
  
 \begin{definition}  Let $\BS$ be a Banach space of analytic functions on $\D$ satisfying Definition~\ref{defBS}. We say a zero-specification set $Z$ for $\BS$ specifies the zeros of a function  $f\in\BS$  provided that the zeros of $f$ and their multiplicities are precisely those specified by $Z$. 
  \end{definition}
  
  Thus, if $\BS$ has boundary-extension type $n$ for some positive integer $n$ and $Z$ having the form (\ref{ZsetP}) specifies the zeros of $f\in \BS$, then for $1 \le i \le l$,  $f$ has $w_i$ as a zero of multiplicity $m_i$, and, for $1\le i \le d$, $f$ has $\zeta_i$ as a zero of multiplicity $\mu_i$. Moreover, these specified zeros are the only zeros of $f$.

 \begin{lemma}  \label{new} Let $\BS$ be a Banach space of analytic functions on $\D$ satisfying Definition~\ref{defBS}, and let $\BS$ have boundary-extension type $n$ for some positive integer $n$.  Suppose that $q$ is a polynomial having no zeros in $\CP\setminus \D^-$, precisely $l$ different zeros $\{w_i: 1 \le i \le l\}\subset \D$ and precisely $d$ different zeros $\{\zeta_i: 1 \le i \le d\}\subset\partial \D$ such that $m_i$ is the multiplicity of $w_i$  for $1\le i \le l$, and $\mu_i\le n$ is the multiplicity of $\zeta_i$ for $1\le i \le d$. Let $Z$ be the corresponding zero-specification set---the set specifying the zeros of $q$, which takes the form (\ref{ZsetP}).  Then the closure of the range of $M_q: \BS\to \BS$ is $S_Z$, which has codimension $\sum_{i=1}^l m_i + \sum_{i=1}^d \mu_i = \deg(q)$ in $\BS$.  \end{lemma} 

\begin{proof}
Suppose that $d=0$. Then $q(z)= c\prod_{i=1}^{l}\big(z-w_i\big)^{m_{i}}$ for some nonzero constant $c$ (where, as usual, we interpret the product to be 1 if $l < 1$).  By Lemma~\ref{MqL} and its proof, $\Ran(M_q) = q\BS = S_Z$, a closed subspace of $\BS$ having codimension $\deg(q)$. 

Now suppose that $d$ is a positive integer. Then
$$q(z)=c\prod_{i=1}^{l}\big(z-w_i\big)^{m_i}\cdot \prod_{i=1}^d (z-\zeta_i)^{\mu_i}\ \text{for some nonzero constant}\ c.$$ 
Clearly, $\Ran(M_q)\subseteq S_Z$, and, because $S_Z$ is closed, we have $\left(\Ran(M_q)\right)^-\subseteq S_Z$. We will show that $S_Z \subseteq \left(\Ran(M_q)\right)^-$. Set $\tilde{q}(z)=\prod_{i=1}^d (z-\zeta_i)^{\mu_i}$ and $g(z)=c\prod_{i=1}^{l}\big(z-w_i\big)^{m_i}$ (and note  that $q= g\tilde{q}$). 
Let $f\in S_Z$ be arbitrary.   By property (iv) of $\BS$, $f=g\tilde{f}$, where $\tilde{f} \in \BS$.  Note that $\tilde{f}$ lies in the kernel of each  linear functional in $E_T:=\cup_{i=1}^d  \left(\cup_{j=0}^{\mu_i-1}\{E_{\zeta_i}^{(j)}\}\right)$ via an easy argument based on the product rule,  the factorization $f = g\tilde{f}$, $f\in S_Z$, and $g(\zeta) \ne 0$ for all $\zeta \in \partial\D$. 

Because the polynomials are dense in $\BS$, there is a sequence $(p_k)$ of polynomials such that $(p_k)$ converges to $\tilde{f}$ in $\BS$.
Because linear functionals in $E_T$ are continuous on $\BS$, we see that
\begin{equation}\label{bkz}
\text{for}\  1\le i \le d\  \text{and}\ 0\le  j \le \mu_i-1, \   \text{the sequence}\  (p_k^{(j)}(\zeta_i))\ \text{converges to}\ \tilde{f}^{(j)}(\zeta_i) = 0.
\end{equation}
  For each $k\in \Z_{>0}$, let $u_k$ be the Hermite polynomial of degree $\delta:= \left(\sum_{i=1}^d \mu_i\right)-1$ such that
\begin{equation}\label{Heq}
 \text{for integers}\ i \ \text{and}\ j\ \text{satisfying}\ 1\le i \le d\ \text{and}\   0\le j \le \mu_i-1,  u_k^{(j)}(\zeta_i) = p_k^{(j)}(\zeta_i).   
 \end{equation}

By our choice of $u_k$, for each positive integer $k$,  the polynomial $p_k - u_k$  factors as $p_k(z) - u_k(z) = \tilde{q}(z)q_k(z)$  for some polynomial $q_k$;   thus, $p_k - u_k$ belongs to the range of  $M_{\tilde{q}}$.  We have
$$
\big\|p_k - u_k - \tilde{f}\big\|_{\BS} \le \big\|p_k - \tilde{f}\big\|_{\BS} +\big\|u_k\big\|_{\BS}.$$
By our choice of $(p_k)$, we have $\big\|p_k - \tilde{f}\big\|_{\BS}\to 0$ as $k\to \infty$.  We claim that $\left(\big\|u_k\big\|_{\BS}\right)$ also has limit $0$ as $k\to \infty$.  

Assuming our claim is valid,  we see that $\tilde{f}$ is in the closure of the range of $M_{\tilde{q}}$.  Since $M_{g}$ is bounded,  $f = g\tilde{f}$ is in closure of the range of $M_gM_{\tilde{q}}=M_q$; that is, 
$f\in \left(\Ran(M_q)\right)^-$.   Because $f\in S_{Z}$ is arbitrary, we conclude that $S_{Z}\subseteq\left(\Ran(M_q)\right)^-.$  We have already observed that the reverse containment holds.  Thus, $S_{Z} = \left(\Ran(M_q)\right)^-$, assuming the validity of our claim that $\left(\big\|u_k\big\|_{\BS}\right)$ has limit $0$. In the next paragraph, we establish this claim is valid, completing the proof.  

Fix a positive integer $k$. Recall that $u_k(z) := \sum_{i=0}^{\delta} a_{i,k}z^i$ is the Hermite polynomial of degree $\delta= \left(\sum_{i=1}^d \mu_i\right)-1$ satisfying the $\delta + 1$ equations of (\ref{Heq}), with the first $\mu_1$ equations being $u_k(\zeta_1) = p_k(\zeta_1)$, and $u_k^{(j)}(\zeta_1) = p_k^{(j)}(\zeta_1)$ for $0<j \le \mu_1-1$.  Thus, the first equation is $\sum_{i=0}^\delta \zeta_1^i a_{i,k} = p_k(\zeta_1)$ and if $\mu_1 > 1$, the second is $\sum_{i=1}^\delta i\zeta_1^{i-1} a_{i,k} = p_k'(\zeta_1)$.  Let 
$$\mathbf{v}_k =\begin{pmatrix} a_{0,k}\\a_{1,k}\\ \vdots\\a_{\delta, k}\end{pmatrix}\ \text{be the vector of coefficients of}\ u_k\  \text{and}\  \mathbf{b}_k =\begin{pmatrix} p_k(\zeta_1)\\ \vdots\\ p_k^{(\mu_1-1)}(\zeta_1)\\ \vdots\\ p_k(\zeta_d)\\ \vdots\\  p_k^{(\mu_d-1)}(\zeta_d)\end{pmatrix}, $$
and observe that the equations of (\ref{Heq}) can be written in matrix form
$$
\mathbf{M}\mathbf{v}_k = \mathbf{b}_k
$$
where $\mathbf{M}$ is the $(\delta + 1) \times (\delta +1)$ matrix whose rows are determined by the equations of (\ref{Heq}); e.g., the first row of $\mathbf{M}$ has entries $1, \zeta_1, \ldots, \zeta_1^\delta$ and the second, provided $\mu_1 > 1$, has entries $0,  1,  2\zeta_1, \ldots, \delta\zeta_1^{\delta -1}$.  Note well that the rows of $\mathbf{M}$ have no dependence on $k$.  Also, $\mathbf{M}$ is invertible (see, e.g.,  Theorem 2 as well as p.\ 272 of \cite{Dubeau}).  Thus, the coefficients of the polynomial $u_k$ are given by
\begin{equation}\label{ME}
\mathbf{v}_k = \mathbf{M}^{-1}\mathbf{b}_k.
\end{equation}
Observe that every entry of the column vector $\mathbf{b}_k$ converges to $0$ as $k\to \infty$ by (\ref{bkz}).  Because $M^{-1}: \CP^{\delta+1} \to \CP^{\delta +1}$ is continuous, we see from Equation~(\ref{ME}) that the coefficients of $u_k$ converge to $0$ as $k\to \infty$, and it follows from property (vi) of $\BS$ that $\big\|u_k\big\|_{\BS}\to 0$ as $k\to \infty$, which validates our claim and completes the proof.
\end{proof}

We are now in a position to characterize the closed, finite-codimensional subspaces of $\BS$ that are invariant under $M_z$. Each such subspace takes the form $S_Z$ for some finite zero-specification set, and these subspaces are all singly generated: there is a polynomial $q$ such that $S_Z = (q\BS)^-$ and $\deg q$ is the subspace's codimension. 

\begin{theorem} \label{FCIVZS} Suppose that $\BS$ is a Banach space of analytic functions on $\D$ satisfying Definition~\ref{defBS}.  A closed, finite-codimensional subspace $\N$ of $\BS$ is $M_z$ invariant if and only if there is a finite zero-specification set $Z$ for $\BS$  such that $\N = S_Z$.   Moreover, if $\BS$ has boundary-extension type $0$ and $Z$ is a finite zero-specification set for $\BS$; or, if $\BS$ has positive boundary-extension type and $Z$ is a finite zero-specification set for $\BS$ having no specified points in $\partial\D$, then  (i) $S_Z = \Ran(M_q)$ for some polynomial whose zeros are specified by $Z$; and, if $\BS$ has positive boundary-extension type  and $Z$ is a finite zero-specification set for $\BS$ having at least one specified point in $\partial \D$, then (ii) $S_Z = \left(\Ran(M_q)\right)^-$, where $q$ is a polynomial whose zeros are specified by $Z$. In addition, in both situations (i) and (ii), $\deg(q) = \dim\left(\BS/S_Z\right)$.  
\end{theorem}
\begin{proof} We have already noted that for any finite zero-specification set $Z$ for $\BS$,  $S_Z$ is a closed, finite-codimensional subspace of $\BS$.  It is clear that $S_Z$ is invariant under any multiplication operator acting on $\BS$, in particular, invariant under $M_z$.

Now suppose that $\N$ is a closed, finite-codimensional subspace of $\BS$ that is $M_z$ invariant. If $\N = \BS$, then $\N = S_\emptyset$; moreover, $S_\emptyset = \Ran(M_q)$, where $q\equiv 1$, and $\deg(q) = 0$ is the codimension of $S_\emptyset$ in $\BS$.  For the remainder of the proof we assume that $\N$ has positive codimension in $\BS$. 

 Suppose that $\BS$ has boundary-extension type $0$; then by Lemma~\ref{FCIV}, $\N=q\BS$ for some monic polynomial $q$ whose zeros lie in $\D$, and, in this case $\N = S_Z$, where $Z$ specifies the zeros of $q$.  Also, by Lemma~\ref{FCIV}, we have $\deg(q) = \dim(\BS/S_Z)$.

Now assume that $\BS$ has boundary-extension type $n$ for some positive integer $n$.   By Lemma~\ref{CAP}, $\N$ contains a monic polynomial $p$ such that $\deg(p)\le \dim(\BS/\N)$.   By Observation~\ref{MOR} following Definition~\ref{defBS}, the closure of the range of $M_p$ is contained in $\N$.  Note this implies $p$ has positive degree---if $p$ were the constant polynomial $z\mapsto 1$, then $\left(\Ran(M_p)\right)^- = \BS$ implying $\N =\BS$, which contradicts $\dim(\BS/\N) > 0$.   In fact, $p$ must have at least one zero in $\D^-$ because  if not, we  would again have $\N = \BS$ since $M_p$ would be invertible, being a product of multiplication operators $M_{z-\lambda}$ with $|\lambda| > 1$, each of which is invertible by property (iii) of $\BS$. 

 By Lemma \ref{AZOT}, there is a polynomial $q\in \left(\Ran(M_p)\right)^-$ having no zeros outside $\D^-$ and whose zeros in $\D^-$ are precisely the zeros of $p$ in $\D^-$ with the same multiplicities for zeros in $\D$,  and, for zeros in $\partial \D$, with multiplicities being the same for multiplicities less than or equal to $n$ and being equal to $n$ for multiplicities greater than $n$. Again, by Observation~\ref{MOR},  $\left(\Ran(M_q)\right)^-\subseteq \left(\Ran(M_p)\right)^-$, and, because $\left(\Ran(M_p)\right)^- \subseteq \N$, we have $\left(\Ran(M_q)\right)^-\subseteq \N$.  Thus,
\begin{align*}
\dim(\BS/\N) & \le \dim\left(\rule{0in}{0.14in}\BS/\left(\Ran(M_q)\right)^-\right)\\
 & =  \deg (q)  \quad (\text{Lemma}~\ref{new})\\
   & \le \deg(p)\\
   &\le \dim(\BS/\N).
  \end{align*}
 Thus, $\dim(\BS/\N) = \dim(\BS/\left(\Ran(M_q)\right)^-) = \deg(q)$, and, because $\left(\Ran(M_q)\right)^-\subseteq \N$, we have $\left(\Ran(M_q)\right)^- = \N$. Let $Z$ specify the zeros of $q$.  We have   $\N=\left(\Ran(M_q)\right)^- = S_Z$, with the latter equality following from Lemma~\ref{new}; moreover, $\dim(\BS/S_Z) =   \dim(\BS/\N) = \deg(q)$. Finally, if $Z$ has no specified points in $\partial \D$, then by Lemma~\ref{MqL}, $M_q$ has closed range, so that $S_Z = \Ran(M_q)$.
\end{proof}

We now complete the proof of our principal result. 

\begin{theorem} \label{THD}  Suppose that $\BS$ is a Banach space of analytic functions on $\D$ satisfying Definition~\ref{defBS}.
Suppose that $\psi$ is a multiplier of $\BS$. If $M_\psi: \BS \to \BS$ is Fredholm, then $\psi$ is bounded away from $0$ near $\partial\D$.  
\end{theorem}

\begin{proof}
Suppose that $\BS$ has boundary-extension type $0$ and that $M_{\psi}:\BS\to \BS$ is Fredholm. Because $\BS$ is infinite dimensional, $\psi$ is not the zero function.    The range of $M_\psi$ is a closed, finite-codimensional subspace of $\BS$ that is clearly $M_z$ invariant. Thus, by Lemma~\ref{FCIV}, there is a monic polynomial $q$  with all its zeros contained in $\D$ such that $\Ran(M_\psi) = \psi\BS = q\BS$.  Because $1\in \BS$, there is a $g\in \BS$ for which $\psi = qg$ and an $h\in \BS$ such that $\psi h = q$. Thus, $\psi = q g = \psi h g$, and we see $hg \equiv 1$.  It is easy to see that both $g$ and $h$ are multipliers of $\BS$.  For example, let $f\in \BS$ be arbitrary; because $q f\in q \BS = \psi\BS$, there is a $u\in \BS$ such that $q f = \psi u$, which implies $\psi h f=\psi  u$, which implies $hf = u$, so that $hf\in \BS$.  Because $f\in \BS$ is arbitrary, we conclude $h$ is a multiplier of $\BS$. Recall $h = 1/g$ and because $h$ is a multiplier of $\BS$, $h$ is bounded on $\D$, which implies $g$ is bounded away from $0$ on $\D$.  It follows that $\psi = qg$ is bounded away from $0$ near $\partial \D$ because both $q$ and $g$ are.  

Now suppose that $\BS$ has boundary-extension type $n$ for some positive integer $n$ and that $M_{\psi}:\BS\to \BS$ is Fredholm.  Because $\Ran(M_\psi)$ is a closed finite-codimensional subspace of $\BS$ that is invariant under $M_z$, there is a polynomial $q$ having no zeros in $\CP\setminus\D^-$ such that $\left(\Ran(M_q)\right)^- = \Ran(M_\psi)$ by Theorem~\ref{FCIVZS}. 

 We have $q \in \Ran(M_\psi)$, so that $q = \psi g$ for some $g$ in $\BS$.  Let $f\in \BS$ be arbitrary.  Because $\Ran(M_q) \subseteq \Ran(M_\psi)$,  $qf = \psi f_1$ for some $f_1\in \BS$.  Thus, $\psi f_1 = q f = \psi g f$, and we conclude $gf = f_1\in \BS$. Because $f\in \BS$ is arbitrary, we conclude that $g$ is a multiplier of $\BS$.  Again, let $f \in \BS$ be  arbitrary.  Because  $\psi f$  is in the range of $M_\psi$, and   $\left(\Ran(M_q)\right)^- = \Ran(M_\psi)$, there is a sequence $(h_k)$ in $\BS$ such that  $(qh_k)$ converges to $\psi f$ in $\BS$;  equivalently, $\psi g h_k \rightarrow \psi f$ as $k\to \infty$.   However, $M_\psi$ is bounded below, being 1-1 with closed range.  Thus, $gh_k \rightarrow f$ as $k\to \infty$.   Since $f \in \BS$ is arbitrary, we conclude that $M_g$ has dense range.    This means that $g$ has no zeros on the closed  unit disk (since point-evaluation functionals at all points in the closed unit disk are continuous)---recall we are assuming that $\BS$ has positive boundary-extension type. 
  
  Because $g$ has no zeros in $\D^-$ and is continuous on $\D^-$, we see $g$ is a multiplier of $\BS$ that is bounded below on $\D$. Thus, by property (v) of $\BS$, $1/g$ is a multiplier of $\BS$ and we have $q(1/g) = \psi$.     For all $f \in \BS$, we have
$\psi f = q \left(\frac{1}{g} f\right)$ and hence $\Ran(M_\psi) \subseteq \Ran(M_q)$. However, we also have $\Ran(M_q) \subseteq \Ran(M_\psi)$. Thus, $\Ran(M_q) = \Ran(M_\psi)$, and we see $M_q$ is Fredholm.

 Suppose $q$ vanishes at a point $\zeta$ on the unit circle. Then $q(z)=\tilde{q}(z)(z-\zeta)$, where $\tilde{q}$ is a polynomial. Because we are assuming that $\BS$ has boundary-extension type $n$,  there is a sequence $(f_k)$ in $\BS$ such that $(z-\zeta)^{n+1}f_k(z)\rightarrow (z-\zeta)^n$ in $\BS$ as $k\to \infty$. Since $M_{\tilde{q}^{n+1}}$ is bounded on $\BS$, we have 
$$\left(M_{q^{n+1}}f_k\right)(z) = \tilde{q}^{n+1}(z)(z-\zeta)^{n+1} f_k(z)\rightarrow \tilde{q}^{n+1}(z)(z-\zeta)^n\ \text{in}\ \BS \ \text{as}\ k\to \infty. $$
Therefore, the closure of the range of  $M_{q^{n+1}}:\BS\to \BS$ contains $z\mapsto \tilde{q}^{n+1}(z)(z-\zeta)^n$.  Because $M_{q^{n+1}}$ has closed range (it is Fredholm, being a product of Fredholm operators), there is $f\in \BS$ such that $q^{n+1}(z) f(z) = \tilde{q}^{n+1}(z)(z-\zeta)^n$; because $q^{n+1}(z) = \tilde{q}^{n+1}(z)(z-\zeta)^{n+1}$, we conclude that $(z-\zeta)^{n+1}f(z) = (z-\zeta)^n$, so that $f(z) = 1/(z-\zeta)\in \BS$, contradicting the positivity of the boundary-extension type of $\BS$.  It follows that all zeros of the polynomial $q$ belong to $\D$.  We have $\psi = q\frac{1}{g}$ with both $q$ and $1/g$ bounded away from $0$ near $\partial \D$.  Thus, $\psi$ is bounded away from $0$ near $\partial \D$, completing the proof.
\end{proof}

Our work allows us to characterize symbols of Fredholm multiplication operators on $\BS$ and describe the Fredholm index in terms of the number of zeros of the symbol.
 \begin{theorem}  Suppose that $\BS$ is a Banach space of analytic functions on $\D$ satisfying Definition~\ref{defBS}. The multiplication operator $M_\psi: \BS \to \BS$ is Fredholm if and only if there is a polynomial $q$ whose zeros belong to $\D$ and a multiplier $h$ of $\BS$ that is bounded away from $0$ on $\D$ such that $\psi = qh$. Moreover, if $q$ has exactly $i$ zeros counting multiplicities, then the Fredholm index of $M_\psi$ is $-i$.
\end{theorem}
\begin{proof} Suppose that $M_\psi: \BS \to \BS$ is Fredholm.  The proof of Theorem~\ref{THD} establishes that $\psi = qh$ where $q$ is a polynomial  whose zeros belong to $\D$ and $h$ is a multiplier of $\BS$ that is bounded away from $0$ on $\D$. 

Conversely, suppose that $\psi = qh$, where $q$ is a polynomial whose zeros belong to $\D$ and $h$ is a multiplier of $\BS$ that is bounded away from $0$ on $\D$.  Then by property (v) of $\BS$, $1/h$ is also a multiplier of $\BS$, and we have  $qf =  qh \left(\frac{1}{h} f\right) = \psi \left(\frac{1}{h} f\right)$ for all $f\in \BS$, so that $\Ran(M_q)\subseteq \Ran(M_\psi)$.  Conversely, $\psi f = q(hf)$ for all $f\in \BS$, so that $\Ran(M_\psi) \subseteq \Ran(M_q)$. Thus, $\Ran(M_\psi) = \Ran(M_q)$, and, by Lemma~\ref{MqL},  $\Ran(M_\psi)=\Ran(M_q) = q\BS$ is a closed subspace of $\BS$ having finite codimension. Hence, $M_\psi:\BS\to\BS$ is Fredholm, as desired.   Moreover, by Lemma~\ref{MqL}, we have $\dim\left(\rule{0in}{0.14in}\BS/\Ran(M_\psi)\right) = \dim\left(\rule{0in}{0.14in}\BS/\Ran(M_q)\right) = \deg(q)$.  Because $M_\psi$ is injective, the Fredholm index of $M_\psi$ is $-\deg(q)$ and $\deg(q)$ is the number of zeros of $q$ counting multiplicities; thus, the description of Fredholm index provided by the theorem is valid. 
\end{proof}

 We conclude this section with two corollaries of our work.  For a function $f\in H(\D)$, let $Cl(f, \partial\D)$ be the cluster set of $f$ relative to $\partial \D$, i.e., the set of points $w$ in the extended complex plane for which there is a sequence $(z_n)$ in $\D$ with $(|z_n|)$ converging to $1$ such that $(f(z_n))$ converges to $w$. 

\begin{cor} Suppose that $\BS$ is a Banach space of analytic functions on $\D$ satisfying Definition~\ref{defBS}. Then $\sigma_e(M_\psi)$,  the essential spectrum  of $M_\psi: \BS \to \BS$, is  $Cl(\psi, \partial\D)$, so that if $\BS$ has positive boundary-extension type, then $\sigma_e(M_\psi) = \psi(\partial \D)$.
\end{cor}
In the context of multiplication operators on the Hardy space $H^2(\D)$, the preceding characterization of the essential spectrum is due to Douglas (see, e.g., \cite[Proof of Corollary 7.37]{DBAT}), and a different proof is provided by Deddens and Wong \cite[Proposition 2]{DW}.  For the Bergman spaces $A^p(\D)$, Axler  provides the characterization,  which follows from a more general characterization \cite[Theorem 23]{AFC} as discussed in Section~\ref{introsec}. For the Hardy-Sobolev spaces $H^2_\beta(\D)$, (equivalently, the $D_\alpha$ spaces), Cao, He, and Zhu provide the characterization in Lemmas 10 and 13 of \cite{CHZ}.  

 \begin{cor} Suppose that $\BS$ is a Banach space of analytic functions on $\D$ satisfying Definition~\ref{defBS}. If $\BS$ has boundary-extension type $0$, then a closed, finite-codimensional subspace $\N$ of $\BS$  is $M_z$ invariant if and only if $\N$ is the range of a Fredholm multiplication operator on $\BS$.   However, if $\BS$ has positive boundary-extension type, then a closed, finite-codimensional subspace $\N$ of $\BS$ that is $M_z$ invariant is either the range of a Fredholm multiplication operator having polynomial symbol necessarily bounded away from $0$ near $\partial\D$ or the closure of the range of a multiplication operator having polynomial symbol that is not bounded away from $0$ near $\partial \D$. 
 \end{cor}

 \section{Examples of Banach spaces $\BS$ satisfying Definition~\ref{defBS}}\label{examples}
 
Throughout this section, $p$ represents a real number greater than or equal to $1$.  

\subsection{The Hardy and weighted Bergman spaces}\label{HWBE}
Recall that the Hardy space $H^p(\D)$ is defined by
$$
H^p(\D) = \left\{f\in H(\D): \|f\|_{H^p(\D)}^p : = \sup_{0 < r < 1}  \frac{1}{2\pi}  \int_0^{2\pi} \left|f\left(re^{i\theta}\right)\right|^p d\theta  <\infty\right\}
$$
and that 
$$
\sup_{0 < r < 1} \frac{1}{2\pi}  \int_0^{2\pi} \left|f\left(re^{i\theta}\right)\right|^p d\theta = \lim_{r\to 1}   \frac{1}{2\pi}    \int_0^{2\pi} \left|f\left(re^{i\theta}\right)\right|^p d\theta =  \frac{1}{2\pi}   \int_0^{2\pi} \left|f\left(e^{i\theta}\right)\right|^p\, d\theta,
$$
where the first equality follows from \cite[Theorem 1.5]{Dur}, the second from \cite[Theorem 2.6]{Dur}, and the integrand of the rightmost integral is the radial-limit function of $f$, defined for almost every $\theta\in [0, 2\pi)$ by $f\left(e^{i\theta}\right) = \lim_{r\to 1^-} f\left(re^{i\theta}\right)$. Excellent references for the Hardy spaces include \cite{Dur,Garn}. 

 Recall that if $\omega > -1$, then the weighted Bergman space $A^p_\omega(\D)$ is defined by
$$
A^p_{\omega}(\D) = \left\{f\in H(\D): \|f\|^p_{A^p_\omega(\D)} :=  \frac{1}{\pi}\int_D |f(z)|^p \left(1-|z|^2\right)^\omega\, dA(z) < \infty\right\},
$$
where $dA$ represents area measure on $\D$.  Using polar coordinates to rewrite the integral defining the norm for $A^p_\omega(\D)$ and then applying the definition of the norm of $H^p(\D)$, we see that  
\begin{equation}\label{HBB}
\|f\|^p_{A^p_\omega(\D)}\le  \frac{1}{1+\omega}  \|f\|^p_{H^p(\D)}\ \text{for} \ f\in H^p(\D).
\end{equation}  Excellent references for the Bergman spaces include \cite{DurSch, HKZ}.
It is well known that the Hardy and Bergman spaces satisfy properties (i)--(vii) of Definition~\ref{defBS}.   We provide verifications here for completeness. 

\begin{itemize}[wide, itemsep=4pt]
\item[(i)] The polynomials are dense in $H^p(\D)$ and $A^p_\omega(\D)$ because for any one of these spaces the dilates $f_r$ of a function in the space converge to $f$ in the space as $r\to 1^-$ (\cite[Theorem 2.6]{Dur}, \cite[Proposition 1.3]{HKZ}), and the Maclaurin series for $f_r$ converges uniformly on $\D^-$ to $f_r$.

\item[(ii)] Point-evaluation functionals on the spaces $H^p(\D)$ and $A^p_\omega(\D)$ are continuous (\cite[Lemma, p.\ ~36]{Dur}, \cite[Proposition 1.1]{HKZ}).

\item[(iii)] It is easy to see that $M_z$ is a contraction on the spaces $H^p(\D)$ and $A^p_\omega(\D)$; thus, (iii) holds.

\item[(iv)] Suppose $q$ is a polynomial having all its zeros in $\D$ and $qf\in H^p(\D)$ for some $f\in H(\D)$.  Because all the zeros of $q$ belong to $\D$, $c: = \min \{|q(\zeta)|: \zeta \in \partial \D\}$ is positive; thus,  $|qf| \ge c|f|$ on the unit circle, from which it follows that $f \in H^p(\D)$. Suppose  $f\in A^p_\omega(\D)$  has a zero at $a\in \D$, so that $f(z) = (z-a)g(z)$ for some $g\in H(\D)$. Then $g$, being continuous on $\{z: |z-a|\le (1-|a|)/2\}$, is bounded on this closed disk, and, off this disk, $|g(z)| \le 2|f(z)|/(1-|a|)$.  It follows that $g\in A_\omega^p(\D)$, and, by induction, (iv) holds for $A^p_\omega(\D)$.

\item[(v)] The integral formulas for the norms of the Hardy spaces $H^p(\D)$ and the weighted Bergman spaces $A^p_\omega(\D)$ make it clear that every bounded function in $H(\D)$ is a multiplier of these spaces. Thus, (v) holds for these spaces.  

\item[(vi)] Let $l$ be a positive integer.   If $(q_k)$ is a sequence of polynomials each of degree at most $l$ and  $q_k(z) = \sum_{j=0}^l a_{j,k}z^j$ is such that for each $j\in\{0,1, \ldots, l\}$  the coefficient sequence $(a_{j,k})_{k=1}^\infty$ converges to $0$, then $(q_k)$ converges to $0$ uniformly on $\D^-$, and, because the norm of $q_k$ in either $H^p(\D)$ or $A^p_\omega(\D)$ is at most $\|q_k\|_\infty$, (vi) holds for these spaces. 

\item[(vii)]  The spaces $H^p(\D)$ and $A^p_\omega(\D)$ are examples of spaces having boundary-extension type $0$. These spaces contain $H^\infty(\D)$ and hence, e.g., include infinite Blaschke products and singular inner functions that do not extend continuously to $\partial \D$. They also include unbounded functions such as $f(z) = \log(1-z)$. Thus, these spaces do not have boundary-extension type $n$ for any $n\in\Z_{>0}$.   To see that these spaces have boundary-extension type $0$, note that if $|\zeta| = 1$, then $f(z) = z-\zeta$ is an outer function, and, for any bounded outer function $f$, $M_f: H^p(\D)\to H^p(\D)$ has a dense range  by Beurling's Theorem, which holds on $H^p(\D)$ by, e.g., \cite[Exercise 18(a), p.\ 94]{Garn}. Hence, for each $\zeta\in \partial\D$, the closure of the range of $M_{z-\zeta}$ on $H^p(\D)$ contains, in particular, $f(z) =1$.   It follows from Inequality~(\ref{HBB}) that for each $\zeta\in \partial \D$ the closure of the range of $M_{z-\zeta}: A^p_\omega(\D) \to A^p_\omega(\D)$ also contains $f(z) = 1$. 
\end{itemize}
\subsection{The Hardy-Sobolev spaces $H^2_\beta(\D)$} 

For $\beta\in \R$, the Hardy-Sobolev  spaces $H^2_\beta(\D)$ are defined in \cite{CHZ} by
$$
H^2_\beta(\D) = \left\{ f\in H(\D):   \|f\|^2_{H^2_\beta(\D)} : =  |f(0)|^2 +  \sum_{k=1}^\infty  k^{2\beta} \left|\hat{f}(k)\right|^2 < \infty \right\},
$$
where $(\hat{f}(k))$ is the sequence of Maclaurin coefficients of $f$.  We work with these spaces using different notation.  For $\alpha \in \R$, let
$$
D_\alpha = \{f\in H(\D): \|f\|^2_\alpha:=  \sum_{k=0}^\infty (k+1)^\alpha \left|\hat{f}(k)\right|^2 <\infty\},
$$
so that $D_\alpha$ is a Hilbert space of analytic functions on $\D$ with inner product 
$$
\langle f, g\rangle_\alpha = \sum_{k=0}^\infty (k+1)^\alpha \hat{f}(k)\overline{\hat{g}(k)}.
$$
Note that if $\beta \le \gamma$, $D_\gamma\subseteq D_\beta$. 
These spaces $D_\alpha$ are the Hardy-Sobolev spaces $H^2_\beta(\D)$ of \cite{CHZ} with $H^2_\beta(\D) = D_{2\beta}$ for all real $\beta$ and with the spaces having equivalent norms.   

For $\alpha >1$, it is easy to check that $\sum_{k=0}^\infty \left|\hat{f}(k)\right| <  \infty$ for $f\in D_\alpha$, so that functions in $D_\alpha$ extend to be continuous on $\D^-$.  Hence,  $D_\alpha$ is contained in the disk algebra for $\alpha >1$; moreover, $D_\alpha$ itself is an algebra for $\alpha > 1$ \cite[Theorem 3]{Kopp}.   Hence, for $\alpha > 1$, the multipliers of $D_\alpha$ are precisely the functions in $D_\alpha$.

 For $\alpha = 0$, observe that $D_\alpha = H^2(\D)$ and $\|\cdot \|_\alpha = \|\cdot \|_{H^2(\D)}$.  For $\alpha < 0$,   $D_\alpha = A^2_{-1-\alpha}(\D)$, with the norm $\|\cdot\|_\alpha$ being equivalent to $\|\cdot\|_{A^2_{-1-\alpha}(\D)}$;  see, e.g.,   \cite[Lemma 2]{Taylor} for a proof.   Finally, $D_1$ is the classical Dirichlet space, and it is natural to view $D_\alpha$ for $0< \alpha \le 1$ as weighted Dirichlet spaces (and this connection perhaps accounts for the use of ``$D$'' in $D_\alpha$).

 We verify that the spaces $D_\alpha$, $\alpha\in \R$, satisfy properties (i)--(vii) of Definition~\ref{defBS}. These properties are known to hold for $D_\alpha$, with the possible exception that $D_\alpha$, for $\alpha > 1$, satisfies the boundary-extension requirement (n.2) of Definition~\ref{BETD} for some $n\in \Z_{>0}$ (see Theorem~\ref{BERC}).  For (iv) and (v), our method of validation may be new.   Most verifications are routine; we include all for completeness. 
\begin{itemize}[wide, itemsep=4pt]
\item[(i)]  Let $\alpha\in \R$. If $f\in D_\alpha$ and $f_n(z) := \sum_{k=0}^{n} \hat{f}(k) z^k$ is its $n$-th Maclaurin polynomial, then $\left\|f_n - f\right\|^2_\alpha = \sum_{k=n+1}^\infty (k+1)^\alpha \left|\hat{f}(k)\right|^2$ has limit $0$ as $n\to \infty$; thus, (i) holds.

\item[(ii)] Observe that if $w\in \D$,  then
$$K_{w}(z): = \sum_{k=0}^\infty \frac{\bar{w}^{k} z^k}{(k+1)^\alpha}\in D_\alpha,$$
and for $f\in D_\alpha$, $\left\langle f, K_{w}\right\rangle_\alpha = f(w)$, from which it follows that (ii) holds.

\item[(iii)] Let $j\in \Z_{>0}$.  It is easy to check that the norm of the operator  $M_{z^j}: D_\alpha \to D_\alpha$ is $1$ for $\alpha \le 0$ and $(j+1)^{\alpha/2}$ for $\alpha > 0$.  Thus, $M_z$ is bounded and its spectral radius  is $1$ ($\lim_{j\to \infty} \left\|(M_z)^j\right\|^{1/j} = \lim_{j\to \infty} \left\|M_{z^j}\right\|^{1/j} = 1$).

\item[(iv)]   We have noted that $D_0 = H^2(\D)$ and for $\alpha < 0$, $D_\alpha$ is a weighted Bergman space (with equivalent norm); thus, (iv) holds for the spaces $D_\alpha$ with $\alpha \le 0$.  Applying Proposition~\ref{pdpDa} below and an induction argument establishes that (iv) holds for all $\alpha$. In the proof of Proposition~\ref{pdpDa}, we use the following lemma whose proof is straightforward. 

\begin{lemma}\label{EL}  For $\alpha \in \R$,  $f\in D_{\alpha}$ if and only if $f'\in D_{\alpha-2}$. 
\end{lemma}

\begin{proposition}\label{pdpDa} Suppose that $D_\alpha$ has property (iv); then $D_{\alpha +2}$ has property (iv).
\end{proposition}
\begin{proof}    Let $q$ be a polynomial all of whose zeros belong to $\D$, and let $qg$ belong to $D_{\alpha + 2}$ for some $g\in H(\D)$.  Because $M_z$ is bounded on $D_{\alpha + 2}$, we see $q^2g\in D_{\alpha+2}$.  Thus, by Lemma~\ref{EL}, $(q^2g)'$ belongs to $D_\alpha$; that is, $q^2g' + \left(2q'\right)qg$ belongs to $D_\alpha$. Note that $\left(2q'\right)qg$ belongs to $D_\alpha$ (in fact it belongs to $D_{\alpha + 2}$ being the polynomial $2q'$ times $qg\in D_{\alpha + 2}$).  Because $q^2g'+ \left(2q'\right)qg$ also belongs to $D_\alpha$, it follows that $q^2g'\in D_\alpha$, and, because we are assuming that $D_\alpha$ has property (iv), and $q^2$ has all its zeros in $\D$, we conclude $g'\in D_\alpha$, so that $g\in D_{\alpha + 2}$ (by Lemma~\ref{EL}). It follows that $D_{\alpha+2}$ has property (iv), as desired. \end{proof} 

\item[(v)]  That property (v) holds for all the spaces $D_\alpha$ follows from Proposition 6 of \cite{CHZ}, which depends on Proposition 4 of \cite{CHZ} whose proof ``is not easy.''  We take a different approach to the proof, obtaining  that (v) holds via an induction argument.

\begin{lemma}\label{DLT2} For $\alpha \le 2$, $D_\alpha$ satisfies property (v) of Definition~\ref{defBS}.
\end{lemma}
\begin{proof}
  We have already seen that $D_\alpha$ satisfies (v) for $\alpha \le 0$ (because $D_0 = H^2(\D)$ and $D_\alpha$ for $\alpha < 0$ is a weighted Bergman space).  Let $0 < \alpha \le 2$.  Let $h$ be a multiplier of $D_\alpha$ that is bounded away from zero on $\D$.  Let $f\in D_\alpha$ be arbitrary.  Because $h$ is a multiplier of $D_\alpha$, we have $hf\in D_\alpha$ and hence, by Lemma~\ref{EL}, $\left(hf\right)'\in D_{\alpha - 2}$. Thus, $hf' + h'f\in D_{\alpha -2}$, and, because $h$ is bounded on $\D$ (being a multiplier of $D_\alpha$) and $f'\in D_{\alpha-2}$ (Lemma~\ref{EL}), we have $hf'\in D_{\alpha-2}$ because $D_{\alpha - 2}$ is either $H^2(\D)$ or a weighted Bergman space and $h$, being bounded,  is a multiplier of the space.  Because both $hf'$ and $hf' + h'f$ belong to $D_{\alpha-2}$, we conclude $h'f\in D_{\alpha-2}$.  
  
Now we show $\frac{1}{h}f$ belongs to $D_\alpha$, which completes the proof of the lemma because $f\in D_\alpha$ is arbitrary.  We have 
  \begin{equation}\label{rmm}
  \left(\frac{1}{h} f\right)'  = \frac{1}{h} f' + \left(-\frac{1}{h^2}\right)h'f.
  \end{equation}
    Because $h$ is bounded away from $0$ on $\D$,  $\frac{1}{h}$ is bounded on $\D$, and it follows that $\frac{1}{h}f'\in D_{\alpha -2}$ because $f'\in D_{\alpha-2}$ and bounded analytic functions are multipliers of this space (either $H^2(\D)$ or a weighted Bergman space).  Because $h'f\in D_{\alpha - 2}$ (by the discussion of the preceding paragraph) and $-\frac{1}{h^2}$ is bounded on $\D$, we conclude  $\left(-\frac{1}{h^2}\right)h'f$ also belongs to $D_{\alpha -2}$.  Thus, we see both summands on the right-hand side of Equation~(\ref{rmm}) belong to $D_{\alpha-2}$, so that  $\left(\frac{1}{h} f\right)'$ does as well. By Lemma~\ref{EL}, $\frac{1}{h}f\in D_\alpha$.      \end{proof} 
    
\begin{lemma}\label{H1L} If $h$ and $g$ both belong to $D_\alpha$ for some $\alpha \ge 0$, then $hg \in D_{-1} = A^2_0(\D)$.
\end{lemma}
\begin{proof} Let $\alpha \ge 0$ and let $h,g \in D_{\alpha}$.  Because $\alpha \ge 0$, we have $D_\alpha \subseteq D_0 = H^2(\D)$. Thus, $hg\in H^1(\D)\subseteq D_{-1} =  A^2_0(\D)$.  That $H^1\D)\subseteq D_{-1}$ follows from Hardy's Inequality \cite[p.\ 48]{Dur} and the boundedness of the Maclaurin coefficients of functions in $H^1(\D)$.  
\end{proof}

\begin{lemma}\label{mr3}  For $2 < \alpha \le 3$, $D_\alpha$ satisfies property (v) of Definition~\ref{defBS}.\end{lemma}
\begin{proof}     Let $2 < \alpha \le 3$, and let $h$ be a multiplier of $D_\alpha$ that is bounded away from zero on $\D$.  Because $1\in D_\alpha$, we see that $h\in D_\alpha$. Let $f\in D_\alpha$ be arbitrary; observe that $f$ is bounded on $\D$ (in fact, in the disk algebra because $\alpha > 1$).  Note that $f'$ and $h'$ belong to $D_{\alpha -2}$ and that $\alpha - 2>0$; thus, $h'f'$ and $\left(h'\right)^2$ belong to $D_{-1}$ by Lemma~\ref{H1L}.  

   We complete the proof that $1/h$ is a multiplier of $D_\alpha$ by showing the second derivative of $\frac{1}{h}f$  belongs to the weighted Bergman space $D_{\alpha-4}$, which implies $\frac{1}{h}f\in D_\alpha$ by Lemma~\ref{EL}.  We have 
  \begin{equation}\label{2de}
  \left(\frac{1}{h} f\right)''  = \frac{1}{h}f'' + 2\left(-\frac{1}{h^2}\right)h'f' +  \left(-\frac{1}{h^2}\right)h''f  +    2\frac{1}{h^3}  \left(h'\right)^2f.
  \end{equation}
  We claim that each of the four summands on the right-hand side of the preceding equation belongs to $D_{\alpha - 4}$, which completes the proof.  The claim holds for the first summand because $f''\in D_{\alpha -4}$ and $\frac{1}{h}$ is bounded on $\D$ and thus is a multiplier of any weighted Bergman space; it holds for the second because $h'f'\in D_{-1}\subseteq D_{\alpha-4}$ and $-1/h^2$ is bounded on $\D$; it holds for the third because $h''\in D_{\alpha-4}$ and both $f$ and $-1/h^2$ are bounded on $\D$; finally, the claim holds for the fourth summand because $(h')^2\in D_{-1} \subseteq D_{\alpha-4}$ while both $f$ and $1/h^3$ are bounded on $\D$. 
  \end{proof}  
  
 We now complete the proof that the $D_\alpha$ spaces satisfy property (v) of Definition~\ref{defBS}.
  
  \begin{proposition}  Let $\alpha\in \R$. If $h$ is a multiplier of  $D_\alpha$ that is bounded away from $0$ on $\D$, then $1/h$ is also a multiplier of $D_\alpha$.  
\end{proposition}

\begin{proof}  We have already established the proposition holds  if $\alpha \le 3$ (Lemmas~\ref{DLT2} and \ref{mr3}).  Our proof that the proposition holds for $\alpha > 3$ is inductive.  As we have just noted, the proposition holds for $1 < \alpha \le 3$.

Let $j$ be a positive integer.  Suppose that for all $\alpha$ satisfying $2j-1 < \alpha \le 2j+1$,  if $h$ is a multiplier of $D_\alpha$ that is bounded away from $0$ on $\D$, then $1/h$ is also a multiplier of $D_\alpha$.  
Let $\alpha$ satisfy $2j+1< \alpha \le 2j+3$, and let $h$ be a multiplier of $D_\alpha$ that is bounded away from $0$ on $\D$.  We show $1/h$ is also a multiplier of $D_\alpha$ to complete the proof.

Both $h$ and $h'$ are multipliers of $D_{\alpha -2}$ because both functions belong to $D_{\alpha -2}$, which is an algebra because    $\alpha -2 >1$.  Now observe that by our induction hypothesis, $1/h$ is a multiplier of $D_{\alpha -2}$, because $h$ is a multiplier of $D_{\alpha -2}$ that is bounded away from $0$ on $\D$. 

Let $f\in D_\alpha$ be arbitrary. We show $\frac{1}{h}f\in D_\alpha$ to complete the proof.  We have
$$
\left(\frac{1}{h} f\right)' = \frac{1}{h} f' + \frac{-1}{h^2} h' f.
$$
The first summand on the right-hand side of the preceding equation belongs to $D_{\alpha -2}$ because $f'\in D_{\alpha -2}$ and $1/h$ is a multiplier of $D_{\alpha -2}$, while second  belongs to $D_{\alpha -2}$ because $h'$ as well as $-1/h^2$ are multipliers of $D_{\alpha -2}$ and $f\in D_{\alpha} \subseteq D_{\alpha -2}$ (alternatively, the second summand belongs to $D_{\alpha-2}$ because it is an algebra and $1/h, h'$, and $f$ belong to $D_{\alpha-2}$).  Thus, $\left(\frac{1}{h} f\right)'\in D_{\alpha -2}$, and by Lemma~\ref{EL}, $\frac{1}{h}f\in D_\alpha$, as desired. 
\end{proof}

\item[(vi)]  Property (vi) of Definition~\ref{defBS} clearly holds for the spaces $D_\alpha$ because for a polynomial $q(z) = \sum_{k=0}^l a_kz^k$ of degree at most $l$, we have $\|q\|^2_\alpha =\sum_{k=0}^l (k+1)^\alpha |a_k|^2$.

\item[(vii)]    The next proposition establishes that $D_\alpha$ has boundary-extension type $0$ if and only if $\alpha \le 1$.  Proposition~\ref{CEPDA} and Theorem~\ref{BERC} below combine to show that if $\alpha$ satisfies $2j-1 < \alpha \le 2j+1$ for some positive integer $j$, then $D_\alpha$ has boundary-extension type $j$. 

   \begin{proposition}\label{DRP} Let $\zeta\in \partial \D$. The operator $M_{z-\zeta}: D_\alpha \to D_\alpha$ has dense range if and only if $\alpha \le 1$. \end{proposition}
\begin{proof} Let $\alpha\le 1$.  Let $f\in D_\alpha$ be in the orthogonal complement of the range of $M_{z-\zeta}$.  Then for all nonnegative integers $k$, we have
$$
\left\langle f, (z-\zeta)z^k\right\rangle_\alpha = 0,
$$
which implies $(k+2)^\alpha\hat{f}(k+1) - \bar\zeta (k+1)^\alpha  \hat{f}(k) = 0$ for every $k\in \Z_{\ge 0}$.  It follows that $\left|\hat{f}(k+1)\right| = \frac{(k+1)^\alpha}{(k+2)^\alpha}\left|\hat{f}(k)\right|$ for $k\in \Z_{\ge 0}$, and, by induction,
\begin{equation}\label{IF}
\left|\hat{f}(k)\right| = \frac{1}{(k+1)^\alpha}\left|\hat{f}(0)\right|\
 \text{for} \ k\in \Z_{\ge 0}.
 \end{equation}
    Because $f\in D_\alpha$ and (\ref{IF}) holds, we have 
$$
\infty >\sum_{k=0}^\infty (k+1)^\alpha \left|\hat{f}(k)\right|^2 = \left|\hat{f}(0)\right|^2 \sum_{k=0}^\infty \frac{1}{(k+1)^\alpha}.
$$
Because $\alpha \le  1$, the preceding inequality implies $\hat{f}(0) = 0$; hence, by Equation~(\ref{IF}), $\hat{f}(k) = 0$ for $k\in\Z_{\ge 0}$, so that $f$ is the zero function.  If follows that the range of $M_{z-\zeta}$ is dense in $D_\alpha$.

Let $\alpha >1$.   For $w\in \D^-$,  observe that the function $K_{w}(z) = \sum_{k=0}^\infty \frac{\bar{w}^{k} z^k}{(k+1)^\alpha}\in D_\alpha,$ and that for $f\in D_\alpha$,
$$
f(w) = \left\langle f, K_w\right\rangle.
$$
Now observe that for each $\zeta\in \partial \D$, $K_\zeta\in D_\alpha$ belongs to the orthogonal complement of the range of $M_{z-\zeta}$; hence, the range of $M_{z-\zeta}$ is not dense in $D_\alpha$, which completes the proof. 
\end{proof}

The next proposition establishes that  if $\alpha$ satisfies  $2j-1 < \alpha \le  2j +1$ for some positive integer $j$,  then $D_\alpha$ satisfies the first of the two requirements for $D_\alpha$ to have boundary-extension type $j$.  

\begin{proposition}\label{CEPDA} If $2j -1 < \alpha \le  2j +1$ for some positive integer $j$, then $D_\alpha \subset C^{j-1}\left(\D^-\right)$.
\end{proposition}
\begin{proof}   We have already noted that if $\alpha > 1$, then $D_\alpha \subset C^0(\D^-)$ (in fact, the Maclaurin coefficients of functions in $D_\alpha$ are absolutely summable when $\alpha > 1$). Thus, in particular, for $1 < \alpha \le 3$, $D_\alpha \subset C^0(\D^-)$.  If $3 < \alpha \le 5$, then every function  $f\in D_\alpha$ also has continuous derivative by the preceding observation because $f' \in D_{\alpha - 2}$  by Lemma~\ref{EL}.  An inductive argument completes the proof.\end{proof}

\begin{theorem}\label{BERC} If $\alpha$ satisfies $2j-1< \alpha \leq 2j+1$ for some positive integer $j$, then for every $\zeta\in \partial \D$, the closure of the range of $M_{(z-\zeta)^{j+1}}: D_\alpha \to D_\alpha$ contains $f(z) = (z-\zeta)^{j}$.
\end{theorem}
\begin{proof}
Let $r$ be a nonnegative number less than $1$, and let $\zeta\in \partial \D$. Observe that $z\mapsto 1/(rz-\zeta)$ belongs to $D_\alpha$ for all $\alpha$.
 Let $g_{r,j}(z) = \frac{(z-\zeta)^{j+1}}{rz-\zeta} - (z-\zeta)^j$,  where $j$ is a positive integer.  For $z \in \D$, we have
 
 \begin{align*}
g_{r,j}(z) &= \frac{(z-\zeta)^{j+1}}{rz-\zeta} - (z-\zeta)^j\\
& = (z-\zeta)^j \left(\frac{z-\zeta}{rz-\zeta} - 1\right)\\
 & = -\bar{\zeta}(z-\zeta)^j\frac{(1-r)z}{1-\bar{\zeta}rz}\\
 & = -\bar{\zeta}(1-r)z(z-\zeta)^j\sum_{n=0}^\infty \bar{\zeta}^n r^n z^n. 
 \end{align*}
 
For $j\in \Z_{>0}$, set $h_{r,j}(z)=(1-r)(z-\zeta)^j\sum_{n=0}^\infty \bar{\zeta}^n r^n z^n$ (so that $g_{r,j}(z) = -\bar{\zeta}z h_{r,j}(z)$).
By induction, we will show that for each positive integer $j$,  
$$h_{r,j}(z)= (1-r)^{j+1}\sum_{n=j-1}^\infty \bar{\zeta}^{n-j+1} r^{n-j+1} z^{n+1}+p_{r,j}(z),$$ 
where 
$$p_{r,j}(z)=-\sum_{k=1}^{j}\zeta (1-r)^k(z-\zeta)^{j-k}z^{k-1}.$$

For $j=1$, we see that 
\begin{align*}
h_{r,1}(z)&=(1-r)(z-\zeta)\sum_{n=0}^\infty \bar{\zeta}^n r^n z^n\\
& = (1-r)\left(\sum_{n=0}^\infty \bar{\zeta}^n r^n z^{n+1} -  \sum_{n=0}^\infty \bar{\zeta}^{n-1} r^n z^n\right)\\
& = (1-r)\left(\sum_{n=0}^\infty \bar{\zeta}^n r^n z^{n+1} -  \sum_{n=0}^\infty \bar{\zeta}^{n} r^{n+1} z^{n+1} - \zeta\right)\\
& = (1-r)\left(\sum_{n=0}^\infty \bar{\zeta}^n r^n(1-r) z^{n+1}  - \zeta\right)\\
& = (1-r)^2\sum_{n=0}^\infty \bar{\zeta}^n r^n z^{n+1} - (1-r)\zeta\\
& = (1-r)^2\sum_{n=0}^\infty \bar{\zeta}^n r^n z^{n+1}+p_{r,1}(z).
\end{align*}

We assume that $h_{r,j}(z)= (1-r)^{j+1}\sum_{n=j-1}^\infty \bar{\zeta}^{n-j+1} r^{n-j+1} z^{n+1}+p_{r,j}(z)$ for some positive integer $j$.
We have 
 \begin{align*}
 h_{r,j+1}(z)&=(z-\zeta)h_{r,j}(z)\\
& =(1-r)^{j+1}(z-\zeta)\sum_{n=j-1}^\infty \bar{\zeta}^{n-j+1} r^{n-j+1} z^{n+1}+(z-\zeta)p_{r,j}(z)\\
&=(1-r)^{j+1}\sum_{n=j}^\infty (1-r)\bar{\zeta}^{n-j} r^{n-j} z^{n+1}-\zeta (1-r)^{j+1}z^j+(z-\zeta)p_{r,j}(z)\\
&=(1-r)^{j+1}\sum_{n=j}^\infty (1-r)\bar{\zeta}^{n-j} r^{n-j} z^{n+1}-\zeta (1-r)^{j+1}z^j-\sum_{k=1}^j \zeta (1-r)^k (z-\zeta)^{j+1-k} z^{k-1}\\
&=(1-r)^{j+2}\sum_{n=j}^\infty \bar{\zeta}^{n-j} r^{n-j} z^{n+1}+p_{r, j+1}(z),
 \end{align*}
 as desired.

Let $j$ be an arbitrary positive integer; we have 
  $$g_{r,j}(z)=-\overline{\zeta}z h_{r,j}(z)= -\overline{\zeta}(1-r)^{j+1}z \sum_{n=j-1}^\infty \bar{\zeta}^{n-j+1} r^{n-j+1} z^{n+1}-\overline{\zeta}zp_{r,j}(z).$$
 Let $\alpha$ satisfy  $2j-1<\alpha\leq 2j+1$, let  $\left\|M_z\right\|$ be the norm of $M_z: D_\alpha \to D_\alpha$, let 
 $$C_{j} = \sup \left\{\frac{(n+j+1)^{2j+1}}{(n+2j+1)\times \cdots \times(n+2)(n+1)}: n\in \Z_{\geq 0}\right\},$$
 and note that the supremum is finite.
  We have
\begin{align*}
\left\|g_{r,j}\right\|_\alpha & \le (1-r)^{j+1}\left\|M_z\right\|\left(\sum_{n=j-1}^\infty (n+2)^{2j+1} r^{2(n-j+1)}\right)^{1/2}+  \left\|M_z\right\|\left\|p_{r,j}\right\|_\alpha\\
& \le (1-r)^{j+1}\left\|M_z\right\|\left(\sum_{n=0}^\infty (n+j+1)^{2j+1} r^{2n}\right)^{1/2}+ \left\|M_z\right\|\left\|p_{r,j}\right\|_\alpha\\
  & \le (1-r)^{j+1}\left\|M_z\right\|\left(C_{j}\sum_{n=0}^\infty (n+2j+1)\times...\times(n+2)(n+1) r^{2n}\right)^{1/2}+ \left\|M_z\right\|\left\|p_{r,j}\right\|_\alpha\\
  &  = (1-r)^{j+1}\left\|M_z\right\| \sqrt{C_{j}} \frac{\sqrt{(2j+1)!}}{(1-r^2)^{j+1}} + \left\|M_z\right\|\left\|p_{r,j}\right\|_\alpha\\
  & = \frac{1}{(1+r)^{j+1}}\left\|M_z\right\| \sqrt{C_{j}} \sqrt{(2j+1)!} + \left\|M_z\right\|\left\|p_{r,j}\right\|_\alpha\\
  & \le \left\|M_z\right\| \sqrt{C_{j}}\sqrt{(2j+1)!} +\left\|M_z\right\|\left\|p_{r,j}\right\|_\alpha\\.
    \end{align*}
   Because $p_{r,j}$ is a polynomial, $\left\|g_{r,j}\right\|_\alpha$ is uniformly bounded for $0\le r < 1$. Therefore,  there is  a sequence $(r_n)$ of numbers in $[0, 1)$ with $\lim_{n\rightarrow \infty}r_n = 1$ such that $(g_{r_n,j})$ is weakly convergent in $D_\alpha$.  Observe that for each $z\in \D$, $g_{r_n,j}(z)$ converges to $0$. Thus, $(g_{r_n,j})$ is weakly convergent in $D_\alpha$ to the zero function; equivalently, the sequence $(f_{n,j})$ in the range of $M_{(z-\zeta)^{j+1}}$ defined by $f_{n,j}(z) = \frac{(z-\zeta)^{j+1}}{r_n z-\zeta}$ converges weakly to $f(z) = (z-\zeta)^j$. Because $f$ is in the weak closure of the subspace $\Ran\left(M_{(z-\zeta)^{j+1}}\right)$, it is also in the norm closure, which completes the proof.
 \end{proof}
 
Remark:  Let $j\in \Z_{>0}$. It is possible to show for each $\alpha$ with $2j-1< \alpha < 2j+1$ (note the strict inequality on the right-hand side) that the $D_\alpha$ norm of $g_{r,j}(z) = \frac{(z-\zeta)^{j+1}}{rz-\zeta} - (z-\zeta)^j$ converges to $0$ as $r\to 1^-$.  
\end{itemize}

\subsection{The derivative Hardy spaces}

For $j\in \Z_{> 0}$,  the $j$-derivative Hardy space $S_j^p(\D)$  consists of functions having $j$-th derivative in $H^p(\D)$ with the norm given by  
\begin{equation}\label{SPJN} 
 \|f\|^p_{S^p_{j}(\D)} =  \sum_{k=0}^{j-1}\left|f^{(k)}(0)\right|^p + \left\|f^{(j)}\right\|^p_{H^p(\D)}.
\end{equation}
The space $S_1^p(\D)$ is typically denoted by $S^p(\D)$ (see, e.g., \cite{Mac}). Various equivalent norms on $S^2(\D)$ are described in \cite{GL}, including the norm that Korenblum used in his paper \cite{Kor} characterizing the closed, invariant subspaces of $M_z$ on $S^2(\D)$. 

Because we are assuming $p \ge 1$, note that a function in $S_j^p(\D)$ is a $j$-order antiderivative of a function in $H^1(\D)$. Thus, all functions in $S_j^p(\D)$ extend continuously to $\D^-$; in fact, derivatives of functions in $S^p_j(\D)$ up to order $j-1$ have continuous extension to $\D^-$.  Thus,  for $f\in S_j^p(\D)$ and integers $k$ satisfying $0 \le k \le j-1$, we have $f^{(k)} \in H^\infty(\D) \subseteq H^p(\D)$.   

The final observation of the preceding paragraph shows $S_i^p(\D)\subseteq S_j^p(\D)$ when $i\ge j\ge 1$ and suggests the following norm for $S_j^p(\D)$:
\begin{equation}\label{MN}
\|f\|^p_{j} := \sum_{k=0}^j \left\|f^{(k)}\right\|^p_{H^p(\D)}.
\end{equation}
The norm $\|\cdot \|_j$ defined above is equivalent to the norm $\| \cdot\|_{S^p_j(\D)}$ with which we choose to work.   Because $|f(0)| \le \|f\|_{H^p(\D)}$ for $f\in H^p(\D)$, we clearly have $\|f\|_{S^p_j(\D)} \le \|f\|_j$ for all $f\in S^p_j(\D)$.   The reverse inequality may be proved by, e.g., iterating an inequality established in the proof of Theorem 2.1 of \cite{HANDHS} showing that if $f^{(k)}\in H^p(\D)$ for some positive integer $k$, then 
\begin{equation}\label{REFEN}
\left\|f^{(k-1)}\right\|_{H^p(\D)} \le 2^{1-1/p}\left(\left|f^{(k-1)}(0)\right|^p + \left\|f^{(k)}\right\|^p_{H^p(\D)}\right)^{1/p}.
\end{equation}
 We provide a different proof of the reverse inequality at the end of this subsection.  
 
 Basic properties of the spaces $S^p_j(\D)$ are described in \cite{MDHS, HANDHS}; for example,  $S_j^p(\D)$ is a Banach space (\cite[Lemma 2.2]{HANDHS}, \cite[Theorem 2.2]{MDHS}). The norm $\|\cdot\|_j$ of (\ref{MN}) for $S^p_j(\D)$ is used in \cite{MDHS} while \cite{HANDHS} uses another norm easily seen to be equivalent to $\|\cdot\|_{S^p_j(\D)}$ of (\ref{SPJN}). 

We verify that the spaces $S_j^p(\D)$ satisfy properties (i)--(vii) of Definition~\ref{defBS}.
\begin{itemize}[wide, itemsep=4pt]
\item[(i)] The density of the polynomials in $S^p_j(\D)$ for $j\in  \Z_{>0}$ follows from \cite[Proposition 2.8]{MDHS}. 

\item[(ii)] The point-evaluation functionals on the spaces $S^p_j(\D)$ are continuous (see, e.g., \cite[Proposition 2.4]{MDHS} or the final paragraph of this section).

\item[(iii)]  That $M_z$ is bounded on $S^p_j(\D)$ and has spectral radius 1 follows from Theorems 3.1 and 3.4 of \cite{MDHS}.

\item[(iv)]  We provide an inductive proof to show that property (iv) of Definition~\ref{defBS} holds for $S_j^p(\D)$ (which is similar to the proof used to show $D_\alpha$ has property (iv)).  We have already noted that if $qf\in H^p(\D)$ for some polynomial $q$ having all its zeros in $\D$ and $f\in H(\D)$, then $f\in H^p(\D)$. Suppose that $qf\in S_1^p(\D)$ for some polynomial $q$ having all its zeros in $\D$ and $f\in H(\D)$.  Then because the polynomial $q$ is a multiplier of $S_1^p(\D)$ (because, e.g.,  $z$ is a multiplier), we have $q^2f\in S_1^p(\D)$.  Thus,
\begin{equation}\label{INHP}
q^2 f' + \left(2q'\right)qf \in H^p(\D).
\end{equation}
 Because $qf\in S_1^p(\D)$, $qf$ is continuous on $\D^-$; thus, $qf \in H^p(\D)$, and $\left(2q'\right) qf \in H^p(\D)$ because $2q'$ is a polynomial. Because the second summand of (\ref{INHP}) belongs to $H^p(\D)$, the first must as well; that is,  $q^2f' \in H^p(\D)$. Because $q^2$ has all its zeros in $\D$, we have $f'\in H^p(\D)$, so that $f\in S^p_1(\D)$, as desired.  
  
  Suppose that for some positive integer $i$, $qf\in S_i^p(\D)$ implies $f\in S_i^p(\D)$ whenever $q$ is a polynomial with all its zeros in $\D$ and $f\in H(\D)$.  Let $q$ be a polynomial having all its zeros in $\D$ and let $qf\in S_{i+1}^p(\D)$ for some $f\in H(\D)$.  Then, $q^2f$ is also in $S_{i+1}^p(\D)$; thus,
  \begin{equation}\label{AIH}
    q^2 f' + \left(2q'\right)qf \in S_i^p(\D).
    \end{equation} 
    Because $qf\in S_{i+1}^p(\D)\subseteq S_i^p(\D)$ and $2q'$ is a polynomial, we see $(2q')qf \in S^p_i(\D)$. Because the second summand of (\ref{AIH}) belongs to $S_i^p(\D)$,  the first must as well; that is,  $q^2f' \in S_i^p(\D)$. Because $q^2$ has all its zeros in $\D$, we have $f'\in S^p_i(\D)$ by the induction hypothesis, and it follows that $f\in S^p_{i+1}$, as desired.  
      
\item[(v)] The space $S_j^p(\D)$ is an algebra \cite[Proposition 2.7]{MDHS}; thus, every element of the space is a multiplier of the space.  We establish that (v) holds for $S_j^p(\D)$ by an inductive argument (which is similar to the one used to show $D_\alpha$ has property (v)).   Suppose that $g$ is a multiplier of $S_1^p(\D)$ that is bounded away from $0$ on $\D$. Let $f\in S_1^p(\D)$ be arbitrary. We claim that $\frac{1}{g}f\in S_1^p(\D)$.  We have
$$
\left(\frac{1}{g} f\right)' = \frac{1}{g} f' + \frac{f}{g^2}\left(-g'\right).
$$
Because both $1/g$ and $\frac{f}{g^2}$ are in $H^\infty(\D)$ while both $f'$ and $-g'$ are in $H^p(\D)$ (because both $f$ and $-g$ belong to $S_1^p(\D)$), we see both summands in the preceding equation belong to $H^p(\D)$. Thus, $\frac{1}{g} f\in S_1^p(\D)$.  Because $f\in S_1^p(\D)$ is arbitrary, we conclude $1/g$ is a multiplier of $S_1^p(\D)$.  

Suppose that for some positive integer $i$, whenever a multiplier of $S_i^p(\D)$ is bounded away from $0$ on $\D$, its reciprocal is also a multiplier of $S_i^p(\D)$.   Let $g$ be a multiplier of $S_{i+1}^p(\D)$ that is bounded away from $0$ on $\D$.  Because $g \in S_{i+1}^p(\D)\subseteq S_i^p(\D)$ and $S_i^p(\D)$ is an algebra, we see $g$ is also a multiplier of $S_i^p(\D)$ that is bounded away from $0$ on $\D$. Thus, $1/g$ is a multiplier of $S^p_i(\D)$ by the induction hypothesis (equivalently, $1/g$ belongs to the algebra $S_i^p(\D)$).   Let $f\in S_{i+1}^p(\D)$ be arbitrary.  Note that  $f \in S_{i+1}^p(\D)\subseteq S_i^p(\D)$. 
We have
$$
\left(\frac{1}{g} f\right)' = \frac{1}{g} f' + \frac{f}{g^2}\left(-g'\right).
$$
Because $1/g, f, f'$, and $g'$ are all members of the algebra $S_i^p(\D)$, both summands on the right-hand side of the preceding equation are in $S_i^p(\D)$, and it follows that $\left(\frac{1}{g} f\right)' \in S_i^p(\D)$.  We conclude that $\frac{1}{g} f\in S^p_{i+1}(\D)$. Because $f\in S_{i+1}^p(\D)$ is arbitrary, $1/g$ is a multiplier of $S_{i+1}^p(\D)$, as desired.  

\item[(vi)]    Let $l$ be a positive integer.  Suppose that $(q_k)$ is a sequence of polynomials each having degree at most $l$, and  $q_k(z) = \sum_{m=0}^l a_{m,k}z^m$ is such that for each $m\in\{0,1, \ldots, l\}$  the coefficient sequence $(a_{m,k})_{k=1}^\infty$ converges to $0$.  Let $j\in \Z_{>0}$ be arbitrary. The coefficient-sequence condition for the sequence $(q_k)$ shows that for all integers $i$ satisfying $0 \le i \le j$, $\left\|q_k^{(i)}\right\|_\infty\to 0$ as $k\to \infty$. It follows that $(q_k)$ converges to $0$ in $S_j^p(\D)$, as desired.  

\item[(vii)] We show that for every $j\in \Z_{>0}$, $S_j^p(\D)$ has boundary-extension type $j$.  We have already observed that $S^p_j(\D) \subset C^{j-1}(\D^-)$.  Proposition~\ref{SD2} below completes the proof. 

\begin{lemma}\label{UBL} Let $0 \le r < 1$.  If $f(z) =z-\zeta$ and $u_r(z) = rz - \zeta$, then  $\|f/u_r\|_{\infty} \le 2$.
\end{lemma}
\begin{proof}  Let $z\in \D$; then
$$
\left|\frac{f(z)}{u_r(z)}\right| = \left| \frac{z-\zeta}{rz -\zeta}\right|\\
   = \left| 1   + \frac{(1-r)z}{rz-\zeta}\right|\\
    \le 1 + \frac{|(1-r)z|}{|\zeta - rz|} \\
     \le 1 + \frac{(1-r)|z|}{1 - r|z|}\\
    \le 2 \quad (|z| < 1) .
    $$
   \end{proof}
   
   The following proposition, together with our observation that $S_1^p(\D) \subset C^0(\D^-)$, shows that $S_1^p(\D)$ has boundary-extension type 1. 
   
    \begin{proposition}\label{SD} Let $p\ge 1$. For $\zeta\in \partial \D$, the closure of the range of the operator $M_{(z-\zeta)^2}: S_1^p(\D)\to S_1^p(\D)$ includes $f(z)= z-\zeta$.
\end{proposition}
\begin{proof}  Let $p > 1$ and let $\zeta\in \partial \D$.  We show $f(z) = z- \zeta$ is in the closure of the range of $M_{(z-\zeta)^2}: S_1^p(\D) \to S_1^p(\D)$.    

Let $0 < r < 1$.  Because $z\mapsto \frac{1}{rz - \zeta}$ has bounded derivative, it belongs to $S_1^p(\D)$. Thus, $h_r(z) := (z-\zeta)^2\frac{1}{rz-\zeta}$ is in the range of $M_{(z-\zeta)^2}$ on $S_1^p(\D)$.  We have  

 \begin{align*} \|h_r-f\|_{S_1^p(\D)}  &=  \left\|(h_r)' - f' \right\|_{H^p(\D)} \quad (\text{note}\ \left(h_r-f\right)(0)= 0)\\
   & =  \left\|\frac{2(z-\zeta)}{rz - \zeta} - \frac{(z-\zeta)^2r}{(rz-\zeta)^2}- 1\right\|_{H^p(\D)}\\
     & =  \left\|\frac{z-\zeta}{rz - \zeta} - \frac{(z-\zeta)^2r}{(rz-\zeta)^2} - \left(1- \frac{z-\zeta}{rz - \zeta}\right)\right\|_{H^p(\D)}\\
    &\le \left\| \frac{z-\zeta}{rz - \zeta}\right\|_\infty  \left\|1 - \frac{(z-\zeta)r}{rz-\zeta}\right\|_{H^p(\D)} + \left\|1- \frac{z-\zeta}{rz - \zeta}\right\|_{H^p(\D)}\\
    &\le 3 \left\|\frac{1-r}{1-r\bar{\zeta} z}\right\|_{H^p(\D)} \quad (\text{Lemma~\ref{UBL}})\\
    & = 3(1-r)\left(\frac{1}{2\pi} \int_0^{2\pi}\frac{1}{\left(\left|1-r\bar{\zeta}e^{i\theta}\right|^2\right)^{p/2}}\, d\theta\right)^{1/p}\\
    &\le  3C(1-r)\left(\frac{1}{(1-r)^{p-1}}\right)^{1/p} (\text{for some}\ C > 0\ \text{by  \cite[Lemma, p.\ 65]{Dur}})\\
    & = 3 C(1-r)^{1/p}.
         \end{align*}
    Thus,  $\displaystyle \lim_{r\to 1^-} \left\|h_r-f\right\|_{S_1^p(\D)} = 0$, and it follows that $f$ is in the closure of range of $M_{(z-\zeta)^2}$ on $S_1^p(\D)$ for every $p > 1$.  Because $\left\|(h_r - f)'\right\|_{H^1(\D)} \le \left\|(h_r - f)'\right\|_{H^2(\D)}$, we conclude that $f$ is also in the closure of range of $M_{(z-\zeta)^2}$ on $S_1^1(\D)$, which completes the proof.
\end{proof}

\begin{proposition}\label{SD2} Let $p\ge 1$ and let $j\in \Z_{>0}$. For $\zeta\in \partial\D$, the closure of the range of the operator $M_{(z-\zeta)^{j+1}}: S^p_{j}(\D)\to S^p_{j}(\D)$ includes $f(z)= (z-\zeta)^{j}$.
    \end{proposition}
    \begin{proof}   Let $0 < r < 1$.  Because $z\mapsto \frac{1}{rz - \zeta}$ has bounded derivatives, it belongs to $S^p_{j}(\D)$ for each positive integer $j$. Thus, for $j\in \Z_{>0}$,  $ (z-\zeta)^{j+1}\frac{1}{rz-\zeta}$ is in the range of $M_{(z-\zeta)^{j+1}}$ on $S^p_j(\D)$.  For $j\in \Z_{>0}$, let $g_{r,j}(z) = \frac{(z-\zeta)^{j+1}}{rz-\zeta} - (z-\zeta)^{j}$.
    
 Let $j$ be a positive integer.  We have
      $$
      \left\|g_{r,j}\right\|^p_{S^p_{j}(\D)} = \sum_{k=0}^{j-1}\left|g^{(k)}_{r,j}(0)\right|^p +  \left\|g^{(j)}_{r,j}\right\|^p_{H^p(\D)}.
      $$   
      We show that $\left\|g_{r,j}\right\|_{S^p_{j}(\D)}$ is uniformly bounded for $0 < r < 1$.  
   
 Note that for $n< j$, we have
     \begin{align*}
     g_{r,j}^{(n)}(z) &=\sum_{k=0}^{n}\binom{n}{k}\left((z-\zeta)^{j+1}\right)^{(n-k)} \bigg(\frac{1}{rz-\zeta}\bigg)^{(k)}-\left((z-\zeta)^{j}\right)^{(n)}\\
     & =\sum_{k=0}^{n}(-1)^k \binom{n}{k} k! r^k\frac{(j+1)!}{(j+1-(n-k))!}(z-\zeta)^{j+1-n+k}\frac{1}{(rz-\zeta)^{k+1}} - \frac{j!}{(j-n)!}(z-\zeta)^{j-n}.
     \end{align*}
   For  $n< j$ and $0 < r < 1$, we obtain from the computation above the following crude overestimate by applying the triangle inequality, using $k!r^k \le n!$ for each term of the initial sum, and replacing the factorials appearing in the denominators with $1$:

   \begin{equation}\label{ZTUB}
     \left|g_{r,j}^{(n)}(0)\right|\leq  n!(j+1)! \sum_{k=0}^n  \binom{n}{k} + j! \le   2^{(j-1)}(j-1)!(j+1)! + j!.
\end{equation}
    
    We claim that for each positive integer $j$, $\left\|g_{r,j}^{(j)}\right\|_{H^{p}(\D)}$  is uniformly bounded for $ 0 < r<1$. By the proof of Proposition~\ref{SD}, $\left\|g'_{r,1}\right\|_{H^{p}(\D)}$ is uniformly bounded for $ 0 < r<1$ (in fact, $\lim_{r\to 1^-} \left\|g'_{r,1}\right\|_{H^{p}(\D)} = 0$). Suppose that for some positive integer $j$,  $\left\|g_{r,j}^{(j)}\right\|_{H^{p}(\D)}$  is uniformly bounded for $0 < r<1$. We see that $g_{r,j+1}(z)=(z-\zeta)g_{r,j}(z)$; thus, we obtain
     \begin{align*}
     g_{r,j+1}^{(j+1)}(z) &=\sum_{k=0}^{j+1}\binom{j+1}{k}(z-\zeta)^{(k)}g_{r,j}^{(j+1-k)}(z)\\
     & =(z-\zeta)g_{r,j}^{(j+1)}(z)+(j+1)g_{r,j}^{(j)}(z).
        \end{align*}
        By the induction hypothesis, there is a constant $C$ such that  $\left\|g_{r,j}^{(j)}\right\|_{H^p(\D)} \leq C$  for $ 0< r<1$. 
     We observe that 
     \begin{align*}
     (z-\zeta)g_{r,j}^{(j+1)}(z)&=(z-\zeta)\sum_{k=0}^{j+1}\binom{j+1}{k}\big((z-\zeta)^{j+1}\big)^{(j+1-k)} \bigg(\frac{1}{rz-\zeta}\bigg)^{(k)}\\
     & =\sum_{k=0}^{j+1}(-1)^k \binom{j+1}{k}r^k (j+1)! (z-\zeta)^{k+1}\frac{1}{(rz-\zeta)^{k+1}}.
     \end{align*}
     By Lemma~\ref{UBL}, $\left\|(z-\zeta)g_{r,j}^{(j+1)}\right\|_{H^p(\D)}\leq \left\|(z-\zeta)g_{r,j}^{(j+1)}\right\|_{\infty}  \le  2^{2j+3}
     (j+1)!$ for $0 < r < 1$. Therefore, $\left\|g_{r,j+1}^{(j+1)}\right\|_{H^{p}(\D)}\le 2^{2j+3}(j+1)! + C(j+1)$ for $ 0 < r<1$, and so, by induction, for each positive integer $j$,   $\left\|g^{(j)}_{r,j}\right\|_{H^p(\D)}$  is  uniformly bounded for $ 0 < r<1$.   Hence,  for every positive integer $j$,  $\left\|g_{r,j}\right\|_{S^p_{j}(\D)}$ is uniformly bounded for $ 0 < r<1$ (because (\ref{ZTUB}) shows that for each nonnegative integer $n<j$, $\left|g_{r,j}^{(n)}(0)\right|$ is uniformly bounded for $0 < r < 1$).
     
    Let  $j$ be an arbitrary positive integer, and let $p > 1$.  Because $\left\|g_{r,j}\right\|_{S^p_{j}(\D)}$  is uniformly bounded for $0 < r<1$ and $S^p_j(\D)$ is a reflexive Banach space \cite[Theorem 2.3]{MDHS},  there is  a sequence $(r_n)$ of numbers in $(0, 1)$ with $\lim _{n \rightarrow \infty}r_n = 1$ such that $(g_{r_n,j})$ is  weakly convergent in $S^p_{j}(\D)$. Note that for each $z \in \D$, $\lim_{n\to \infty} E_{z}(g_{r_n,j})=  0$.
Thus, the weak limit of  $(g_{r_n,j})$ in $S^p_{j}(\D)$ is the zero function.  Equivalently, the sequence $(f_n)$ in the range of $M_{(z-\zeta)^{j+1}}$ defined by $f_n(z) = \frac{(z-\zeta)^{j+1}}{r_n z-\zeta}$ converges weakly to $f(z) = (z-\zeta)^j$ in  $S^p_{j}(\D)$.  Because $f$ is in the weak closure of the subspace $\Ran\left(M_{(z-\zeta)^{j+1}}\right)$, it is also in the norm closure, and we have proved the proposition for $p > 1$. In particular, note that we have established that $f$ is in the norm closure of the range of $M_{(z-\zeta)^{j+1}}$ on   $S^2_j(\D)$.  Since $\|\cdot\|_{S^1_{j}(\D)} \leq  \sqrt{j+1}\, \|\cdot\|_{S^2_{j}(\D)}$, we conclude that $f$ is in the norm closure of  $\Ran\left(M_{(z-\zeta)^{j+1}}\right)$ on $S^1_{j}(\D)$, which completes the proof. 
  \end{proof}
  
  \begin{remark}
 By a longer and  more challenging inductive argument, it is possible to show that  for each $j\in \Z_{>0}$ and $p\ge 1$,   the $S_j^p(\D)$ norm of  $g_{r,j}(z) = \frac{(z-\zeta)^{j+1}}{rz-\zeta} - (z-\zeta)^{j}$ converges to $0$ as $r\to 1^-$. 
\end{remark}  
\end{itemize}

We conclude this subsection by giving an alternative proof of one of the two inequalities that establish the equivalence of the norms (\ref{SPJN}) and (\ref{MN}) by showing there is a constant $\gamma$ such that $\|f\|_j\le \gamma \|f\|_{S^p_j(\D)}$ for all $f\in S_j^p(\D)$.  First, note that if $(f_k)$ is a sequence of functions in $S_j^p(\D)$ converging in the norm $\| \cdot\|_{S_j^p(\D)}$ to the function $f$, then $(f_k^{(j)})$ converges uniformly on compact subsets of $\D$  to $f^{(j)}$ (because $(f_k^{(j)})$ converges to $f^{(j)}$ in $H^p(\D)$).  Because $\left|f_k^{(j-1)}(0)- f^{(j-1)}(0)\right|$ converges to $0$ as $k\to \infty$, we conclude that $f_k^{(j-1)}(z) = \int_0^z f_k^{(j)}(w)\, dw  + f_k^{(j-1)}(0)$ converges uniformly on compact subsets of $\D$ to $f^{(j-1)}(z) = \int_0^z f^{(j)}(w)\, dw + f^{(j-1)}(0)$.  Continuing this integration process if necessary, we conclude that $(f_k)$ converges uniformly on compact subsets of $\D$ to $f$, so that, in particular, the point-evaluation functionals $E_z$ for $z\in \D$ are continuous on $S^p_j(\D)$ with norm $\left\|\cdot \right\|_{S^p(\D)}$.  Because $S^p_j(\D) \subset C^{j-1}(\D^-)$ and the point-evaluation functionals $E_z$, $z\in \D$, are continuous, the Closed-Graph Theorem shows that the identity operator $I: S^p_j(\D) \to C^{j-1}(\D^-)$ is continuous (just as in Remark (ii) following Definition~\ref{BETD}). Thus, there is a constant $\gamma_1$ such that  $\|f\|_{C^{j-1}(\D^-)} \le \gamma_1 \|f\|_{S^p_j(\D)}$ for all $f\in S_j^p(\D)$.  By the definition of $\| \cdot\|_{C^{j-1}(\D^-)}$, we have
$$
\sum_{k=0}^{j-1} \left\|f^{(k)}\right\|_{\infty}  \le \gamma_1 \|f\|_{S^p_j(\D)}, \ \text{for all}\  f\in S_j^p(\D).
$$
Hence, for $f\in S^p_j(\D),$
\begin{align*}
\|f\|_j & = \left(\sum_{k=0}^{j-1}\left\|f^{(k)}\right\|^p_{H^p(\D)} + \left\|f^{(j)}\right\|^p_{H^p(\D)}\right)^{1/p}\\
         & \le \sum_{k=0}^{j-1}\left\|f^{(k)}\right\|_{H^p(\D)} + \left\|f^{(j)}\right\|_{H^p(\D)}\\
         & \le  \sum_{k=0}^{j-1}\left\|f^{(k)}\right\|_\infty + \|f\|_{S^p_j(\D)}\\
         & \le (\gamma_1 +1) \|f\|_{S^p_j(\D)}.
          \end{align*}
 Thus, $\|f\|_j\le \gamma   \|f\|_{S^p_j(\D)}$, where $\gamma = \gamma_1 + 1$, as desired.

\subsection{Additional examples}
Additional examples of Banach spaces of analytic functions on $\D$ satisfying requirements (i)--(vii) of Definition~\ref{defBS} include the disk algebra $A$, consisting of all analytic functions on $\D$ extending continuously to $\D^-$ and endowed with the supremum norm $\|\cdot\|_\infty$.   The space $A$ has boundary-extension type $1$: by definition, $A\subset C^0(\D^-)$, and for each $r\in [0, 1)$ and $\zeta\in \partial\D$, we have
$$
\left\|\frac{(z-\zeta)^2}{rz - \zeta} - (z-\zeta)\right\|_\infty = (1-r)\left\|\frac{z(z-\zeta)}{rz - \zeta}\right\|_\infty \le 2(1-r),
$$
where we have used Lemma~\ref{UBL} to obtain the inequality.  Verifying that properties (i)--(vi) of Definition~\ref{defBS} hold for $A$ is straightforward. 

 It is also easy to verify, using essentially the same arguments as those in Section~\ref{HWBE} for standard-weight Bergman spaces, that all radial-weight Bergman spaces  (see, e.g.,  \cite{KM})  satisfy Definition~\ref{defBS}.  These  radial-weight Bergman spaces $A^p_G(\D)$ are defined by
 $$
 A_G^p(\D) = \left\{f\in H(\D): \|f\|^p_{A^p_G(\D)} : = \int_\D |f(z)|^p\, G\left(|z|\right)\, dA(z)< \infty\right\},
 $$
 where $G$ is a positive continuous function on $(0,1)$ such that  $\int_0^1 G(r)r\, dr <\infty$. These spaces have boundary-extension type $0$.
 
       Another family of Banach spaces for which Definition~\ref{defBS} holds is the Dirichlet-space family $D^p(\D)$, $p\ge 1$, defined by
     $$
     D^p(\D) =\left\{f\in H(\D): \|f\|^p_{D^p(\D)} : = |f(0)|^p + \left\|f'\right\|^p_{A^p(\D)}\right\},
     $$
     where $A^p(\D)$ is the unweighted Bergman-$p$ space.  Note that $D^2(\D)$ is $D_\alpha$ with $\alpha = 1$, where the norms of the two spaces are equivalent.  Verifying that the spaces $D^p(\D)$ (for $p \ge 1$) satisfy properties (i)--(vi) of Definition~\ref{defBS} is straightforward based on the methods already used in this section.  For property (vii), there is a dichotomy: if $1\le p \le 2$, then the space $D^p(\D)$ has boundary-extension type $0$; however, if $p > 2$, the space $D^p(\D)$ has boundary-extension type $1$.  We provide a proof below. 
     
  Let $1 < p \le 2$, let $\zeta\in \partial \D$, and let $t\in (0,1)$.  We have 
     \begin{align*}
      \left\|\frac{z-\zeta}{tz-\zeta} - 1 \right\|_{D^p(\D)} & = (1-t) \left\|\frac{-\zeta}{(tz-\zeta)^2}\right\|_{A^p(\D)}\\
      &  = (1-t)\left( \frac{1}{\pi} \int_\D  \frac{1}{|tz-\zeta|^{2p}}\, dA(z)\,  \right)^{1/p}\\
       &   = (1-t) \left( \frac{1}{\pi} \int_0^1\int_0^{2\pi}\frac{1}{\left(\left|1 - tr\bar{\zeta}e^{i\theta}\right|^2\right)^p}\,  r\, d\theta \, dr\right)^{1/p}\\
         & \le C (1-t) \left(\int_0^1 \frac{1}{(1 - rt)^{2p-1}}\,  dr\right)^{1/p} (\text{for some}\ C >0\  \text{by \cite[Lemma, p.\ 65]{Dur}})\\
        & \le \frac{C}{t^{1/p}(2p-2)^{1/p}} (1-t)^{2/p -1}.
        \end{align*}
            
        For $1 < p < 2$, the inequality above yields $\left\|\frac{z-\zeta}{tz-\zeta} - 1 \right\|_{D^p(\D)} \to 0$ as $t\to 1^-$, so that for these spaces, we have shown $D^p(\D)$ has boundary-extension type $0$.  For $p =2$, the inequality above shows that $\left\|\frac{z-\zeta}{tz-\zeta} - 1 \right\|_{D^p(\D)}$ is uniformly bounded for $0 < t < 1$, and a weak-convergence argument shows that $1$ belongs to the closure of the range of $M_{z-\zeta}: D^2(\D) \to D^2(\D)$. (Alternatively, that $D^2(\D)$ has boundary-extension type $0$ follows from Proposition~\ref{DRP} because $D^2(\D)$ is $D_\alpha$ for $\alpha =1$.)  Finally, if $p =1$, then the final line of the inequality above becomes ``$\le C (1-t)\log(1/(1-t))/t$,'' and we have $\left\|\frac{z-\zeta}{tz-\zeta} - 1 \right\|_{D^1(\D)} \to 0$ as $t\to 1^-$, as desired.

       Now, we show that $D^p(\D)$ has boundary-extension type $1$ for $p > 2$.  First, we note that for all $p\ge1$ and all $z\in \D$, we have
    \begin{equation}
    |f(z)| \le \frac{\|f\|_{A^p(\D)}}{(1-|z|)^{2/p}};
    \end{equation}
         see, e.g, \cite[Theorem 2.1]{ZSHF}.  Thus, if $f'\in A^p(\D)$ for some $p> 2$, then $|f'(z)| \le \frac{\left\|f'\right\|_{A^p(\D)}}{(1-|z|)^{1-\alpha}}$, where $\alpha := 1-2/p$ satisfies $0 < \alpha < 1$; thus, $f$ extends to be continuous on $\D^-$ by a result due to Hardy and Littlewood (\cite[Theorem 5.1]{Dur}).  To see that $D^p(\D)$ has boundary-extension type $1$ for $p > 2$, let $p > 2$, let $\zeta\in \partial\D$, let $t\in (0,1)$, and, as in the proof of Proposition~\ref{SD}, let $h_t(z) = (z-\zeta)^2/(tz - \zeta)$ and $f(z) = z-\zeta$.    
         We have
         $$\left\|h_t - f\right\|_{D^p(\D)} = \left\|h_t' - f'\right\|_{A^p(\D)} \le \left\|h_t' - f'\right\|_{H^p(\D)} \le 3C(1-t)^{1/p}
         $$
         for some $C > 0$ by the proof of Proposition~\ref{SD}.  Thus, for $p > 2$, $D^p(\D)$ has boundary-extension type $1$, as claimed.  

    \section{Questions} 
    
  We have shown that properties (i)--(vii) of Definition~\ref{defBS} are sufficient to ensure that a multiplication operator on $\BS$  is Fredholm if and only if its symbol is bounded away from $0$ near $\partial \D$ (where $\BS$ is a Banach space of analytic functions on $\D$ satisfying (i)--(vii)). To what extent are the properties of Definition~\ref{defBS} necessary for this characterization to be valid?   Sheldon Axler has pointed out (private communication) that certain closed, $M_z$-invariant subspaces of Bergman spaces furnish examples relevant to this question.  For instance, Theorem 6.1 of \cite{HRS} shows that there exists a closed, $M_z$-invariant subspace $\mathcal{M}$ of $A^p(\D)$ ($p\ge 1$), such that $\dim(\mathcal{M}/(z\mathcal{M})) = \infty$.  Thus, $M_z: \mathcal{M}\to\mathcal{M}$ is not Fredholm on $\mathcal{M}$ yet $\psi(z) = z$ is bounded away from $0$ near $\partial \D$.  It is not difficult to check that $\mathcal{M}$   satisfies properties  (ii), (iii), (v), and (vi) of Definition~\ref{defBS} (with (vi) holding vacuously);  also, it is possible to show that (vii) holds: $\mathcal{M}$ has boundary-extension type $0$ in the sense that $M_{z-\zeta}: \mathcal{M}\to\mathcal{M}$ has dense range for each $\zeta\in \partial\D$.   Suppose that $B$ is a Banach space of analytic functions on $\D$ satisfying only (ii)--(v) of Definition~\ref{defBS} and that $\psi$ is a multiplier  of $B$; by the proof of Theorem~\ref{TED} (and the remark after it),  the following implication is valid:
 \begin{equation}\label{EasyD}
\psi\ \text{bounded away from}\ 0\ \text{near}\ \partial \D  \implies  M_\psi: B\to B\ \text{is Fredholm.}
\end{equation}
 Thus, the example $M_z:\mathcal{M}\to \mathcal{M}$ discussed above shows that the inclusion of (iv) in the list (ii)--(v) of properties of $B$ is necessary for the implication (\ref{EasyD}) to hold.

   To what extent can our work be generalized to multiplication operators on Banach spaces of analytic functions on planar domains $\Omega$?   In particular,  if (a) we replace $\D$ in Definition~\ref{defBS} by a bounded planar domain $\Omega$ such that no connected component of $\partial\Omega$ is equal to a point, (b) we modify the spectral-radius assumption for $M_z$ (e.g. $M_z - \lambda I$ is invertible when $\lambda\in \CP\setminus \Omega^-$), and (c) we replace $\BS$ by $\BS(\Omega)$, then are Fredholm multiplication operators on $\BS(\Omega)$ precisely those whose symbols are bounded away from $0$ near $\partial \Omega$?  The preceding question is motivated by Corollary 6 of \cite{AB}, which shows that for such domains $\Omega$, a multiplication operator on the Bergman-$p$ space $A^p(\Omega)$, $p\ge1$, is Fredholm if and only if its symbol is bounded away from $0$ near $\partial \D$.
   
   What conditions on a Banach space $B$ of analytic functions on $\Omega$ will ensure that all closed finite-codimensional, $M_z$-invariant subspaces of $B$ are singly generated?  The question of single generation (cyclicity of $M_z$ restricted to the subspace) for a specific $M_z$-invariant subspace is raised on page 817 of \cite{AB}.

\bibliographystyle{amsplain}

\end{document}